\documentclass[a4paper]{article}
\usepackage[left=20mm, right=20mm, top=20mm, bottom=30mm]{geometry}
\usepackage{amsmath, amssymb, amsthm}
\usepackage{mathtools}
\usepackage{bm}
\usepackage{graphicx}
\usepackage{hyperref}
\usepackage{enumitem}
\usepackage{color}
\usepackage{orcidlink}
\usepackage{thmtools}
\usepackage{cleveref}
\usepackage{booktabs}
\usepackage{multirow}
\usepackage{cellspace}
\usepackage{makecell}
\usepackage{float}

\hypersetup{
    pdftitle={Overcoming slow Kolmogorov width decay in parametric optimal control via neural network surrogates},
    pdfauthor={Hendrik Kleikamp, Martin Lazar, Juan Ricardo Mu\~{n}oz},
    pdfkeywords={Parametrized optimal control problems, Kolmogorov barrier, nonlinear surrogates, U-Net architecture, a posteriori error estimation}
}

\usepackage{tikz}
\usetikzlibrary{arrows.meta}

\usepackage{pgfplots}
\pgfplotsset{compat=1.18}
\usepgfplotslibrary{groupplots}

\definecolor{viridisYellow}{RGB}{253,231,37}
\definecolor{viridisGreen}{RGB}{94,201,98}
\definecolor{viridisTeal}{RGB}{33,145,140}
\definecolor{viridisBlue}{RGB}{59,82,139}
\definecolor{viridisViolet}{RGB}{68,1,84}
\definecolor{matchingRed}{HTML}{D01F3C}
\definecolor{matchingOrange}{HTML}{FFC077}

\colorlet{unet}{matchingRed}
\colorlet{paramdec}{viridisYellow}
\colorlet{cnnaerom}{viridisBlue}
\colorlet{rbrom}{viridisViolet}
\colorlet{cnnae}{viridisGreen}
\colorlet{orthproj}{matchingOrange}

\usepackage{authblk}
\usepackage[maxbibnames=99, backend=bibtex]{biblatex}
\def\ph{{\varphi}}

\def\R{{\mathbb{R}}}

\def\mA{{\mathbf{A}}}
\def\mI{{\mathbf{I}}}
 
\def\lA{{\mathcal{A}}}
\def\lH{{\mathcal{H}}}
\def\lM{{\mathcal{M}}}

\newcommand{\params}{{\mathcal{P}}}

\hypersetup{
    colorlinks=true,
    linkcolor=blue,
    citecolor=blue,
    urlcolor=blue
}

\newtheorem{theorem}{Theorem}[section]

\newtheorem{assumption}[theorem]{Assumption}

\newtheorem{corollary}[theorem]{Corollary}

\theoremstyle{remark}
\newtheorem{remark}[theorem]{Remark}
\Crefname{remark}{Remark}{Remarks}

\renewcommand{\d}[1]{\,\mathrm{d}#1}

\DeclareMathOperator*{\argmin}{arg\,min}

\newcommand{\cnnlayer}[7]{
    \draw[fill=#7, fill opacity=0.25]
        (#1-#4/2.,#2-#5/2.,#3+#6/2.) -- (#1+#4/2.,#2-#5/2.,#3+#6/2.) -- (#1+#4/2.,#2+#5/2.,#3+#6/2.) -- (#1-#4/2.,#2+#5/2.,#3+#6/2.) -- cycle
        (#1-#4/2.,#2+#5/2.,#3+#6/2.) -- (#1+#4/2.,#2+#5/2.,#3+#6/2.) -- (#1+#4/2.,#2+#5/2.,#3-#6/2.) -- (#1-#4/2.,#2+#5/2.,#3-#6/2.) -- cycle
        (#1+#4/2.,#2-#5/2.,#3-#6/2.) edge[dashed] (#1-#4/2.,#2-#5/2.,#3-#6/2.)
        (#1-#4/2.,#2-#5/2.,#3-#6/2.) edge[dashed] (#1-#4/2.,#2+#5/2.,#3-#6/2.)
        (#1-#4/2.,#2-#5/2.,#3-#6/2.) edge[dashed] (#1-#4/2.,#2-#5/2.,#3+#6/2.);
}

\newcommand{\cnnlayerend}[7]{
    \draw[fill=#7, fill opacity=0.25]
        (#1-#4/2.,#2-#5/2.,#3+#6/2.) -- (#1+#4/2.,#2-#5/2.,#3+#6/2.) -- (#1+#4/2.,#2+#5/2.,#3+#6/2.) -- (#1-#4/2.,#2+#5/2.,#3+#6/2.) -- cycle
        (#1+#4/2.,#2-#5/2.,#3+#6/2.) -- (#1+#4/2.,#2-#5/2.,#3-#6/2.) -- (#1+#4/2.,#2+#5/2.,#3-#6/2.) -- (#1+#4/2.,#2+#5/2.,#3+#6/2.) -- cycle
        (#1-#4/2.,#2+#5/2.,#3+#6/2.) -- (#1+#4/2.,#2+#5/2.,#3+#6/2.) -- (#1+#4/2.,#2+#5/2.,#3-#6/2.) -- (#1-#4/2.,#2+#5/2.,#3-#6/2.) -- cycle
        (#1+#4/2.,#2-#5/2.,#3-#6/2.) edge[dashed] (#1-#4/2.,#2-#5/2.,#3-#6/2.)
        (#1-#4/2.,#2-#5/2.,#3-#6/2.) edge[dashed] (#1-#4/2.,#2+#5/2.,#3-#6/2.)
        (#1-#4/2.,#2-#5/2.,#3-#6/2.) edge[dashed] (#1-#4/2.,#2-#5/2.,#3+#6/2.);
}

\newcommand{\cnnarrow}[6]{
    \draw[->,>=Latex] (#1,#2,#3) -- (#4,#5,#6);
}

\title{Overcoming slow Kolmogorov width decay in parametric optimal control via neural network surrogates\thanks{This publication is based upon work from COST Action InterCoML, supported by COST (European Cooperation in Science and Technology).}}
\author{Hendrik Kleikamp\,\orcidlink{0000-0003-1264-5941}\,\thanks{IDea\_Lab, University of Graz, Leechgasse~34, 8010~Graz, Austria, ({\tt hendrik.kleikamp@uni-graz.at}).}\quad Martin Lazar\,\orcidlink{0000-0002-4034-5770}\,\thanks{Faculty of Electrical Engineering and Applied Computing, University of Dubrovnik, Ćira Carića~4, 20~000~Dubrovnik, Croatia, ({\tt mlazar@unidu.hr}).}\quad Juan Ricardo Mu\~{n}oz\,\orcidlink{0009-0003-9788-792X}\,\thanks{Department of Mathematical Engineering, Faculty of Physical and Mathematical Sciences, Universidad de Chile, Av.~Beauchef~851, 8370458~Santiago, Chile, ({\tt juan.munoz@dim.uchile.cl}).}\textsuperscript{\, ,\,}\thanks{Faculty of Electrical Engineering and Applied Computing, University of Dubrovnik, Ćira Carića~4, 20~000~Dubrovnik, Croatia.}\textsuperscript{\, ,\,}\thanks{Corresponding author, email: {\tt juan.munoz@dim.uchile.cl}}}
\date{July 14, 2026}

\begin{document}

\maketitle

\begin{abstract}
In this paper we deal with parametric, linear-quadratic optimal control problems in which the solution can be uniquely characterized by the optimal final time adjoint state.
As a motivating example, we establish theoretical results showing that for distributed control of the heat equation, the manifold of final time adjoints over the parameter space exhibits a slow decay of its Kolmogorov width if this was already the case for the parameter-dependent target states.
Traditional linear reduced-order models would thus require a large reduced space in order to guarantee a sufficient accuracy, making them inefficient in this application.
To overcome the limitation of linear models, we discuss a nonlinear surrogate based on U-Nets that maps parametric fields to approximate final time adjoints.
We show that a suitable a posteriori error estimator remains applicable to the U-Net approximation and can be used to certify the surrogate results.
Through two extensive numerical experiments, we show the potential of the U-Net surrogate and compare it with several linear and nonlinear methods from the literature.
The results show that the U-Net consistently achieves the highest accuracy among the methods considered while requiring significantly fewer training samples.
\end{abstract}


\section{Introduction}
This contribution is concerned with optimal control problems that depend on a set of parameters, for instance influencing target states, physical quantities of the dynamical system or initial conditions.
Specifically, we consider linear-quadratic optimal control problems with a term penalizing deviations from the terminal target state and a term measuring the control energy over time.
The Hilbert uniqueness method~(HUM)~\cite{lions1988controlabilite1} allows us to characterize the optimal state, control, and adjoint entirely through the corresponding parameter-dependent final time adjoint.
Consequently, rather than approximating the optimal control or state directly, we focus on constructing surrogate models for the final time adjoint, from which all remaining optimal quantities can be recovered.
In many applications such as inverse problems, uncertainty quantification or parameter optimization, one is interested in solving the parametric problem for a large number of different parameter values.
Since the high-fidelity full-order model~(FOM) is typically computationally expensive, it is tremendously costly to solve the~FOM for all relevant parameters.
To make such settings feasible, one therefore resorts to surrogate or reduced-order models~(ROMs) that provide approximate solutions under significantly smaller costs.
\par
Classical linear~ROMs rely on a suitable low-dimensional subspace of the high-dimensional solution space and solve a reduced problem in the subspace.
We refer to~\cite{quarteroni2007reduced,dede2012reduced,lazar2016greedy,ballarin2022spacetime,kleikamp2025twostage} for examples of such methods in the context of parametric optimal control problems.
Combinations of linear approaches with machine learning algorithms have been considered recently as well, see~\cite{kleikamp2025greedy}.
An overview of methods for parametric control problems can be found in~\cite{JL21}.
However, the applicability of linear~ROMs is limited when it comes to problems where the solution manifold cannot be approximated well with low-dimensional subspaces, i.e.~when the Kolmogorov~$N$-width of the manifold (measuring the approximation error of the manifold with the best possible linear subspace of dimension~$N$) decays slowly in the reduced dimension~$N$.
\par
This limitation has motivated the development of nonlinear approximation techniques.
Many of these approaches employ modern machine learning architectures, such as convolutional autoencoders, to obtain nonlinear representations of the solution manifold, see for instance~\cite{fresca2021comprehensive,franco2023deep}.
As pointed out in~\cite{duriez2017machine,bertsekas2019reinforcement,zuazua2022control,cerf2023combining}, control theory and machine learning are closely related and share many ideas, which motivates the application of machine learning methods for control problems, particularly in parametric settings.
For instance, the authors in~\cite{demo2023extended} extended physics informed neural networks to parametric optimal control problems.
In~\cite{verma2025neural}, different neural network-based methods for learning the value function or the policy in a parametric setting have been proposed and compared to a model-based approach.
Deep reinforcement learning combined with dictionary learning to obtain sparse control policies has been applied to parametric~PDE control in~\cite{botteghi2024parametric}.
\par
The main contributions of our paper are twofold.
First, we analyze how structural properties of the target states, including regularity and parameter-dependent shifts, are transferred to the corresponding final time adjoints.
In particular, we show that slow Kolmogorov~$N$-width decay of the target manifold is inherited by the manifold of final time adjoints.
This explains why low-dimensional linear reduced-order models are not effective for the transport-dominated parametric problems considered in this work.
\par
As a second contribution, motivated by the theoretical analysis, we propose a nonlinear surrogate model to overcome the~Kolmogorov~$N$-width limitations.
More precisely, we consider a~U-Net architecture that approximates the solution operator associated with the final time adjoint equation by mapping parameter-dependent fields directly to the corresponding final time adjoints.
Additionally, since the surrogate approximates the final time adjoint itself, the residual-based a posteriori error estimator remains applicable and can be used to evaluate the quality of the predicted solution.
Finally, we discuss the design decisions for the neural network architecture and provide extensive numerical experiments comparing the proposed approach with several existing linear and nonlinear surrogate models on benchmark problems that are particularly challenging for traditional linear reduced-order models. 
\newline\newline
\noindent The paper is organized as follows: In~\Cref{sec:parametric-lqocp} we introduce the parametric optimal control problem considered in this work. Afterwards, we discuss in~\Cref{sec:LROM_lim} reduced basis methods and their limitations by exemplarily characterizing the Kolmogorov width of the manifold of final time adjoints in terms of the manifold of target states. We further analyze the structure of the final time adjoints for an example with moving targets and distributed controls. The~U-Net based surrogate for parametric optimal control problems is then introduced and examined in~\Cref{sec:nonlinear-surrogate}. We compare the approach to several other well-known methods from the literature, including linear reduced-order models and convolutional autoencoders, on different numerical examples, see~\Cref{sec:numerical-experiments}. As experiments, we first consider the same setting that already served as our guiding theoretical example (\Cref{ssec:HGC}). Secondly, we apply all methods including the~U-Net approach to a more complicated test case with a localized control region, a discontinuous and parameter-dependent diffusion field and a target state with strong parameter dependence (\Cref{sec:ball-example}). Finally, we conclude the paper in~\Cref{sec:conclusion-outlook} by summarizing our main findings and discussing directions for future work.

\section{Parametric linear-quadratic optimal control problems}\label{sec:parametric-lqocp}
We begin by introducing the abstract framework for a class of parametrized linear-quadratic optimal control problems.
\par
Let~$\lH$ and~$\mathcal U$ be real Hilbert spaces endowed with scalar products~$\langle\cdot,\cdot\rangle_{\lH}$ and~$\langle\cdot,\cdot\rangle_{\mathcal U}$, respectively, and associated norms~$\|\cdot\|_{\lH}$ and~$\|\cdot\|_{\mathcal U}$. When no ambiguity arises, we simply write~$\langle\cdot,\cdot\rangle$ and~$\|\cdot\|$. We denote by~$\mathcal L(\mathcal U,\lH)$ the space of bounded linear operators from~$\mathcal U$ to~$\lH$, and we set~$\mathcal L(\lH)\coloneqq\mathcal L(\lH,\lH)$. Throughout, $\lH$ is the state space and~$\mathcal U$ is the control space.
\par
Let~$\params$ be the parameter set. For each~$\mu\in\params$, we consider the linear control system
\begin{equation}\label{eq:PDOC_1}
    \left\{
    \begin{aligned}
        \dot{y}_\mu(t) &= A_\mu y_\mu(t)+B_\mu u_\mu(t) & \text{for }t\in[0,T],\\
        y_\mu(0) &= y_\mu^0,
    \end{aligned}
    \right.
\end{equation}
where~$A_\mu\colon D(A_\mu)\subset \mathcal H \to \mathcal H$ generates a~$C_0$-semigroup~$(e^{tA_\mu})_{t\geq0}$ on~$\mathcal H$, and~$B_\mu\in\mathcal L(\mathcal U,\lH)$ is a parameter-dependent control operator. Moreover, $y_\mu^0\in\lH$ is the initial datum, $y_\mu^T\in\lH$ is the target state, and~$T>0$ is the final time.
\par
We impose the following standing assumptions ensuring well-posedness for each fixed parameter.
\begin{assumption}\label{as:param_dep}
    The set~$\params$ is a compact subset of a Banach space. For each~$\mu\in\params$, $A_\mu$ generates a~$C_0$-semigroup~$(e^{tA_\mu})_{t\geq0}$ on~$\mathcal H$.  We assume that, for every~$z\in\mathcal H$, the map~$\mu\mapsto e^{tA_\mu}z$ is uniformly continuous for~$t\in[0,T]$. 
Moreover, the mappings~$\mu\mapsto B_\mu$, $\mu\mapsto y_\mu^0$ and~$\mu\mapsto y_\mu^T$ are continuous.
\end{assumption}

\begin{remark}
    We emphasize that the present framework is slightly more general than the one considered in~\cite{kleikamp2025greedy}. Assuming that each~$A_\mu$ generates a~$C_0$-semigroup, rather than requiring~$A_\mu\in\mathcal L(\mathcal H)$ and continuity of the map~$\mu\mapsto A_\mu$, allows us to also include unbounded operators in the study, such as elliptic differential operators.
\end{remark}

For any control~$u\in L^2(0,T;\mathcal U)$, we denote by~$y_\mu\colon[0,T]\to\lH$ the corresponding solution of~\eqref{eq:PDOC_1}.
Our objective is to steer the system close to the target at time~$T$, that is~$y_\mu(T)\approx y_\mu^T$.
To determine such a control, we introduce a quadratic cost functional. In this simplified setting, we penalize the final state deviation and the control energy, and we fix a scalar weight~$\alpha>0$. For each~$\mu\in\params$, we define
\begin{align*}
    \mathcal J_\mu(u)
    =
    \frac12
    \left(
    \alpha \lVert y_\mu(T)-y_\mu^T\rVert_\lH^2
    +
    \int_0^T \lVert u(t)\rVert_{\mathcal U}^2\d{t}
    \right).
\end{align*}
This functional is continuous, quadratic and strongly convex with respect to~$u$, and therefore admits a unique minimizer.
\par
With this notation, the parametrized optimal control problem reads
\begin{align}\label{eq:PDOC_2}
    \min_{u\in L^2(0,T;\mathcal U)}\ \mathcal J_\mu(u)
    \quad\text{subject to}\quad
    \dot{y}_\mu(t)=A_\mu y_\mu(t)+B_\mu u(t),
    \quad y_\mu(0)=y_\mu^0.
\end{align}
For each~$\mu\in\params$, this problem is well posed and admits a unique optimal control~$u_\mu^*$, with associated state~$y_\mu^*$.
\par
For completeness, we now recall the optimality system (see~\cite[Theorem~2.4]{kleikamp2025greedy}) associated to this problem. The optimal solution can be characterized through an adjoint state~$\varphi_\mu^*\colon[0,T]\to\lH$ satisfying
\begin{equation}\label{eq:opt_sys_a}
    \left\lbrace
    \begin{aligned}
        \dot{y}_\mu(t) &= A_\mu y_\mu(t)+B_\mu u_\mu(t),\\
        -\dot{\varphi}_\mu(t) &= A_\mu^*\varphi_\mu(t),\\
        u_\mu(t) &= -B_\mu^*\varphi_\mu(t),
    \end{aligned}
    \right.
\end{equation}
for~$t\in[0,T]$, together with
\begin{align}\label{eq:opt_sys_b}
    y_\mu(0)=y_\mu^0,
    \qquad
    \varphi_\mu(T)=\alpha \big(y_\mu(T)-y_\mu^T\big).
\end{align}
This system shows that the solution is entirely determined by the final time adjoint~$\varphi_\mu^*(T)\in\lH$. In particular, we can express all variables in terms of~$\varphi_\mu^*(T)$. First, we have
\begin{align*}
    \varphi_\mu^*(t)
    =
    e^{(T-t)A_\mu^*}\varphi_\mu^*(T),
\end{align*}
which directly yields the control
\begin{align*}
    u_\mu^*(t)
    =
    -\,B_\mu^*e^{(T-t)A_\mu^*}\varphi_\mu^*(T).
\end{align*}
Substituting this expression into the state equation, we obtain
\begin{align*}
    y_\mu^*(t)
    =
    e^{tA_\mu }y_\mu^0
    -
    \int_0^t e^{(t-s)A_\mu}B_\mu B_\mu^*e^{(T-s)A_\mu^*}\varphi_\mu^*(T)\d{s}.
\end{align*}
Evaluating at~$t=T$, we obtain
\begin{align}\label{eq:op.st.T}
    y_\mu^*(T)
    =
    e^{TA_\mu }y_\mu^0-W_\mu(T)\,\varphi_\mu^*(T),
\end{align}
where the controllability Gramian is defined by
\begin{align}\label{eq:w_gram}
    W_\mu(T)
    =
    \int_0^T e^{(T-s)A_\mu}B_\mu B_\mu^*e^{(T-s)A_\mu^*}\d{s}.
\end{align}
The operator~$W_\mu(T)\in\mathcal{L}(\lH)$ is bounded, self-adjoint, and positive semidefinite.
Indeed, for every~$v\in\lH$,
\begin{align}\label{equ:positive-semidefinite-Gramian}
    \langle W_\mu(T) v,v\rangle
    =
    \int_0^T \|B_\mu^*e^{(T-s)A_\mu^*}v\|_{\mathcal U}^2\d{s}
    \geq 0.
\end{align}
Finally, combining~\eqref{eq:op.st.T} with the terminal condition in~\eqref{eq:opt_sys_b}, we arrive at the equation for the final time adjoint:
\begin{align}\label{eq:final_adjoint_system}
    \left(I+\alpha W_\mu(T)\right)\varphi_\mu^*(T)
    =
    \alpha \left(e^{T A_\mu }y_\mu^0-y_\mu^T\right).
\end{align}
This identity shows that the parametrized optimal control problem reduces to solving, for each~$\mu\in\params$, a linear equation in~$\lH$ involving the controllability Gramian. Notice that Assumption~\ref{as:param_dep} implies that the map~$\mu\longmapsto \varphi_\mu^*(T)$ is continuous from~$\params$ into~$\mathcal H$.

\begin{remark}\label{rem:Simpl}
    Due to the linearity of the control system, without loss of generality the initial data~$y^0_\mu$ can be replaced by zero. The equivalent optimal control problem~\eqref{eq:PDOC_2}  is obtained by replacing the target~$y_\mu^T$ with~$e^{TA_\mu }y_\mu^0-y_\mu^T$.
    For this reason and in order to simplify the presentation, in the sequel we  assume~$y^0_\mu=0$.    
\end{remark}
The full-order model~(FOM) is solved by computing the terminal adjoint~$\varphi^*_\mu(T)$ from~\eqref{eq:final_adjoint_system}. For large-scale discretizations, the controllability Gramian~$W_\mu(T)$ is not assembled explicitly. Instead, its action on a vector is evaluated through forward and adjoint time integrations, and the resulting linear system is solved with the conjugate gradient~(CG) method. Each~CG iteration requires at least one Gramian application and therefore several time-stepping solves. While this procedure is feasible for a small number of parameter values, it becomes computationally demanding in a multi-query setting, where the same computation must be repeated for many parameters~$\mu\in\params$. Consequently, repeated full-order simulations and iterative linear solves become the dominant computational cost, motivating the use of surrogate models.
\par
The cost increases further in the parametric setting because the system matrix~$I+\alpha W_\mu(T)$ depends on the parameter through the underlying dynamics. As a result, a different linear system must be solved for each query. Since every~CG iteration requires one Gramian application, and each Gramian application involves the solution of one forward and one adjoint evolution problem, the dominant cost of a full-order solve is given by
\begin{align*}
    \mathcal O\!\left(n^3+m^3+n_{\mathrm{CG}}\,n_t\,n\,(n+m)\right),
\end{align*}
where~$n_{\mathrm{CG}}$ is the number of~CG iterations, $n_t$ is the number of time steps, $n$ is the state dimension, and~$m$ is the control dimension (see~\cite[Section~5.2.1]{kleikamp2025greedy} for additional details). In typical applications, it holds~$n_t\geq n$. Therefore, when many parameter queries are required, the overall computational effort scales linearly with the number of queries and at least cubic in~$n$, making the repeated solution of the full-order optimality system the main computational bottleneck.
\par
To address this issue, we consider surrogate models that approximate the optimal final time adjoint~$\varphi_\mu^*(T)$ for a given parameter~$\mu\in\params$. As discussed above, the optimal control, adjoint, and state are uniquely determined by the final time adjoint through~\eqref{eq:final_adjoint_system}. Therefore, the objective is to construct a surrogate that provides an accurate approximation of~$\varphi_\mu^*(T)$ at a significantly lower computational cost than solving the full-order system, while still allowing the accuracy of the approximation to be assessed through a posteriori error estimates.

\section{Linear reduced-order models and their limitations}\label{sec:LROM_lim}
In this section we analyze the application of linear reduced-order models when treating parametric linear-quadratic optimal control problems  of the form presented above.
We shall show that even in a very simplified model setting this approach faces severe, intrinsic limitations. This makes application of  reduced models based on linear subspaces inefficient and motivates to strive for nonlinear approaches in~\Cref{sec:nonlinear-surrogate}.

\subsection{Linear reduced basis method for approximation of final time adjoints}\label{sec:linear-reduced-model}
In this section, we briefly describe the linear reduced-order model for final time adjoints that was originally introduced in~\cite{lazar2016greedy}.
The reduced-order model will serve as a reference in our numerical experiments, see~\Cref{sec:numerical-experiments}.
\par
Let us assume we are given a (low-dimensional) reduced subspace~$\lH_N\subset\lH$ of dimension~$N\in\mathbb{N}$ with~$N\ll\dim\lH$.
We would like to approximate the optimal final time adjoint~$\varphi_\mu^*(T)\in\lH$ by a suitable reduced solution~$\tilde{\varphi}_\mu^N\in\lH_N$.
The optimal final time adjoint is characterized according to~\eqref{eq:final_adjoint_system} by
\begin{align*}
    \left(I+\alpha W_\mu(T)\right)\varphi_\mu^*(T)
    =
    \alpha  y_\mu^T,
\end{align*}
where we, without loss of generality, assume that the initial datum~$y_\mu^0$ is equal to zero (cf.~\Cref{rem:Simpl}). 
\par
In order to obtain a reasonable reduced approximation, $\tilde{\varphi}_\mu^N$ should fulfill the same linear system as accurately as possible.
The reduced solution is therefore defined as
\begin{align*}
    \tilde{\varphi}_\mu^N = \argmin_{p\in\lH_N}\, \left\lVert \left(I+\alpha W_\mu(T)\right)p - \alpha y_\mu^T\right\rVert_\lH^2,
\end{align*}
i.e.~by minimizing the residual of the linear system over the reduced space~$\lH_N$.
In practice, it is possible to compute~$\tilde{\varphi}_\mu^N$ by selecting a basis~$\varphi_1,\ldots,\varphi_N\in\lH_N$ of~$\lH_N$ (a so-called reduced basis), computing~$y_\mu^i = \left(I+\alpha W_\mu(T)\right)\varphi_i$ for~$i=1,\ldots,N$ and solving the linear system
\begin{align*}
    \sum\limits_{i=1}^{N}\gamma_i\cdot y_\mu^i = P_{\operatorname{span}(y_\mu^1,\ldots,y_\mu^N)}\left(\alpha  y_\mu^T\right)
\end{align*}
for the coefficients~$\gamma=[\gamma_i]_{i=1}^{N}\in\R^N$, where~$P_{\operatorname{span}(y_\mu^1,\ldots,y_\mu^N)}$ denotes the orthogonal projection onto the linear subspace~$\operatorname{span}(y_\mu^1,\ldots,y_\mu^N)\subset\lH$.
The reduced solution is then given as
\begin{align*}
    \tilde{\varphi}_\mu^N = \sum\limits_{i=1}^{N}\gamma_i\cdot \varphi_i.
\end{align*}
We remark at this point that the main computational effort in this reduced-order model is required to obtain~$y_\mu^1,\ldots,y_\mu^N$ by applying the controllability Gramian, i.e.~solving two dynamical systems, at the costs of
\begin{align*}
    \mathcal O\!\left(n^3+m^3+N\,n_t\,n\,(n+m)+N^2\,n+N^3\right),
\end{align*}
see~\cite[Section~5.2.2]{kleikamp2025greedy}.
It is therefore beneficial to choose the dimension~$N$ of the reduced space as small as possible in order to achieve a large speedup compared to the full-order simulation.
\par
As shown in~\cite{kleikamp2025greedy}, the norm of the residual of the linear system is a reliable and efficient a posteriori error estimator for the error in the final time adjoint, i.e.~we have for any approximate final time adjoint~$p\in\lH$ that
\begin{align*}
    \left\lVert\varphi_\mu^*(T)-p\right\rVert_\lH\ \leq\ \eta_\mu(p)\ \leq\ \left\lVert I+\alpha W_\mu(T)\right\rVert_{\mathcal{L}(\lH)}\left\lVert\varphi_\mu^*(T)-p\right\rVert_\lH.
\end{align*}
The error estimator~$\eta_\mu\colon\lH\to[0,\infty)$ is for~$p\in\lH$ defined as
\begin{align}\label{eq:error-estimator}
    \eta_\mu(p)\coloneqq\left\lVert \left(I+\alpha W_\mu(T)\right)p-\alpha  y_\mu^T\right\rVert_{\lH}.
\end{align}
To evaluate the error estimator, it is essentially sufficient to apply the operator~$I+\alpha W_\mu(T)$ to the approximate final time adjoint.
This allows for an efficient error estimation without the need to compute the exact final time adjoint by solving the~FOM.
It is important to note here that the error estimator is not restricted to the reduced solution~$\tilde{\varphi}_\mu^N\in\lH_N$ but is applicable to any~$p\in\lH$ with the same guarantees of a reliable and efficient estimate.
\par
Based on the error estimate~$\eta_\mu(p)$ it is also possible to bound the error in the control, for instance by
\begin{align*}
    \lVert u_\mu^*-\tilde{u}_\mu^p\rVert_{L^2([0,T];\mathcal{U})}^2 = \int\limits_{0}^{T} \lVert u_\mu^*(t)-\tilde{u}_\mu^p(t)\rVert_\mathcal{U}^2\d{t} = \langle \varphi_\mu^*(T)-p,W_\mu(T)(\varphi_\mu^*(T)-p)\rangle \leq \lVert W_\mu(T)\rVert_{\mathcal{L}(\lH)}\cdot\eta_\mu^2(p),
\end{align*}
where~$\tilde{u}_\mu^p\colon[0,T]\to\mathcal{U}$ denotes the control induced by the approximate final time adjoint~$p\in\lH$, i.e.~$\tilde{u}_\mu^p=-B_\mu^*e^{(T-t)A_\mu^*} p$, see also~\eqref{equ:positive-semidefinite-Gramian}. The error in the state at final time can further be bounded by
\begin{align*}
    \lVert x_\mu^*(T)-x_\mu^p(T)\rVert_\lH = \lVert W_\mu(T)\varphi_\mu^*(T)-W_\mu(T)p\rVert_\lH \leq \lVert W_\mu(T)\rVert_{\mathcal{L}(\lH)}\cdot\eta_\mu(p),
\end{align*}
where~$x_\mu^p\colon[0,T]\to\lH$ denotes the state trajectory associated to the final time adjoint~$p$, computed by solving~\eqref{eq:opt_sys_a}.
Consequently, if the upper bound of the Gramian norm is on disposal, the residual-based a posteriori error estimator provides estimates for control and state errors as well.

\subsection{Study of the Kolmogorov width  decays}\label{sec:transport-dominated-effect-moving-target}
The efficiency of linear reduced basis methods is directly related to the Kolmogorov widths of the solution manifold.
We recall that, given a compact subset~$\lM\subset \lH$, its Kolmogorov~$N$-width is defined by
\begin{align}\label{eq:Kolmo_w}
    d_N(\lM)
    \coloneqq
    \inf_{\substack{V_N\subset \lH\\ \dim(V_N)=N}}
    \sup_{v\in\lM}
    \inf_{w\in V_N}
    \|v-w\|_{\lH}.
\end{align}
This quantity measures the best possible error achieved by approximating the manifold~$\lM$ using a linear space of dimension~$N$. Therefore, the decay rate of~$d_N(\lM)$ determines the efficiency of linear reduced-order approximation methods.
\par
The final time adjoint plays a crucial role in the class of parametric optimal control problems we consider in this contribution.
In particular, for reduced-order models such as the one introduced in~\Cref{sec:linear-reduced-model}, it is important to understand how the final time adjoint varies with the parameter.
We therefore analyze in the following a particular case where certain properties of the parametrized input data (such as the target~$y_\mu^T$)  are transferred to the final time adjoint and thus identify scenarios in which reduced-order models relying on low-dimensional linear subspaces are insufficient to capture the solution manifold~$\lM_\ph$ of final time adjoints well.
In particular, we aim to relate the Kolmogorov widths of the manifold of terminal target states
\begin{align}\label{eq:Man_Tar}
    \lM_T \coloneqq\{y_\mu^T:\mu\in\params\}
\end{align}
to those of the solution manifold
\begin{align}\label{eq:Man_Fin}
    \lM_\varphi\coloneqq  \{\varphi_\mu^*(T):\mu\in\params\}.
\end{align}
In order to analyze to which extent the structural properties of the target manifold~\eqref{eq:Man_Tar} are transferred to the manifold of final time adjoints~\eqref{eq:Man_Fin}, we consider two special cases presented in the next subsections. 
\begin{remark}
    As already indicated in the previous section, the assumption of zero initial datum~$y_\mu^0=0$ does not impose any restrictions on the system under consideration. Indeed, in a general case all the analysis can be conducted along the same lines just by replacing the final targets manifold~$\lM_T$ by
   ~$\{y_\mu^T- e^{T A_\mu}y_\mu^0:\mu\in\params\}$.
\end{remark}

\subsubsection{Parameter-independent Gramian}
We consider a simplified setting in which the following assumption holds:
\begin{assumption}
    In this subsection, we assume that both the dynamics~$A$ and the control operator~$B$ are parameter-independent, resulting in a parameter-independent Gramian~$W_\mu(T) = W(T)$.
\end{assumption}

We first recall the well-known result providing an invariance property of the Kolmogorov width~\eqref{eq:Kolmo_w} under a bounded linear isomorphism, see~\cite[Theorem~1.1]{Pinkus1985}. For the sake of completeness, we present its brief proof. 

\begin{theorem}\label{th:abs_transfer}
    Let~$\lH$ be a Hilbert space, let~$\lM\subset\lH$ be a compact subset, and assume that~$\lA \in\mathcal{L}(\lH)$ with bounded inverse. Then, for every~$N\in\mathbb N$,
    \begin{align}\label{eq:width_transfer}
        \frac{1}{\|\lA^{-1}\|_{\mathcal L(\lH)}}
        \,d_N(\lM)
        \leq
        d_N(\lA\lM)
        \leq
        \|\lA\|_{\mathcal L(\lH)}
        \,d_N(\lM).
    \end{align}
\end{theorem}

Equivalently, if there exist constants~$c_T,C_T>0$ such that
\begin{align*}
    c_T\|z\|_{\lH}
    \leq
    \|\lA z\|_{\lH}
    \leq
    C_T\|z\|_{\lH},
    \qquad
    \text{for all }  z\in\lH,
\end{align*}
then
\begin{align*}
    c_T d_N(\lM)
    \leq
    d_N(\lA\lM)
    \leq
    C_T d_N(\lM).
\end{align*}

\begin{proof}
Let~$V_N\subset\lH$ be an arbitrary subspace with~$\dim(V_N)=N$. Since~$\lA$ is linear and invertible, the set~$\lA^{-1}V_N$ is also an~$N$-dimensional subspace of~$\lH$. Hence, for every~$u\in\lM$,
\begin{align*}
    \inf_{v\in V_N}
    \|\lA u-v\|_{\lH}
    &=
    \inf_{w\in \lA^{-1}V_N}
    \|\lA u-\lA w\|_{\lH}
    \\
    &\leq
    \|\lA\|_{\mathcal L(\lH)}
    \inf_{w\in \lA^{-1}V_N}
    \|u-w\|_{\lH}.
\end{align*}
Taking the supremum over~$u\in\lM$ and then the infimum over all~$N$-dimensional subspaces~$V_N$, we obtain
\begin{align*}
    d_N(\lA\lM)
    \leq
    \|\lA\|_{\mathcal L(\lH)}
    d_N(\lM).
\end{align*}
For the reverse inequality, we apply the previous estimate to the set~$\lA\lM$ and to the operator~$\lA^{-1}$. Since
\begin{align*}
    \lM
    =
    \lA^{-1}(\lA\lM),
\end{align*}
it follows that
\begin{align*}
    d_N(\lM)
    \leq
    \|\lA^{-1}\|_{\mathcal L(\lH)}
    d_N(\lA\lM),
\end{align*}
which yields the lower bound in~\eqref{eq:width_transfer}.
\end{proof}

As a consequence, we may apply the previous result to the manifold of final time adjoints.
\begin{corollary}
    There exist constants~$c_T,C_T>0$ such that
    \begin{align*}
        c_T d_N(\lM_T)
        \leq
        d_N(\lM_\varphi)
        \leq
        C_T d_N(\lM_T),
        \qquad
        \text{for all }  N\in\mathbb N.
    \end{align*}
    In particular, $d_N(\lM_\varphi)$ and~$d_N(\lM_T)$ are equivalent, that is, the Kolmogorov width is the same up to multiplicative constants independent of~$N$.
\end{corollary}
\begin{proof}
    Introducing the operator 
    \begin{align}\label{eq:At}
        \lA_T=I + \alpha W(T),
    \end{align}
    the relation~\eqref{eq:final_adjoint_system} reduces to
    \begin{align*}
        \lA_T\varphi_\mu^*(T)
        =
        -\alpha y_\mu^T.
    \end{align*}
    Therefore,
    \begin{align*}
        \lA_T\lM_\varphi
        =
        -\alpha \lM_T.
    \end{align*}
    Since~$W(T)$ is self-adjoint and positive semidefinite and~$\alpha >0$, for every~$z\in \lH$ we have, 
    \begin{align*}
        \left\langle \lA_T z,z\right\rangle_{\lH} = \|z\|_{\lH}^2 + \alpha \left\langle W(T) z\, , \, z \right\rangle_{\lH} \geq \|z\|_{\lH}^2.
    \end{align*}
    Therefore, $\lA_T$ is coercive. Since~$\lA_T$ is also bounded, the Lax-Milgram theorem implies that the operator~$\lA_T$ is invertible and~$\|\lA_T^{-1}\|_{\mathcal{L}(\lH)}\leq 1$, that is, $\lA_T$
    is boundedly invertible.
    \par
    This allows applying~\Cref{th:abs_transfer} with the operator~$\lA_T$ and  the manifold~$\lM_\varphi$, which yields the claim.
\end{proof}

The last result implies that if the target manifold~$\lM_T$ has slowly decaying Kolmogorov widths, then the manifold of final time adjoints~$\lM_\varphi$ inherits the same decay rate up to multiplicative constants.
Therefore, for instance the poor approximability of transport-dominated target manifolds, see~\cite{ohlberger2016reduced,greif2019decay,arbes2025kolmogorov}, is transferred to the corresponding family of final time adjoints.
In the following, we will consider target states that are given as fixed profiles that are shifted depending on the parameter, similar to initial conditions moving over time in transport or wave equations.
The Kolmogorov width of such solution manifolds has been investigated in the aforementioned references and is of the order of~$\mathcal{O}\!(N^{-1/2})$ for discontinuous profiles.
Hence, to reach a maximum approximation error of~$\varepsilon>0$, a suitable reduced space of dimension at least~$\varepsilon^{-2}$ (up to constants) has to be considered.

\subsubsection{Globally distributed control}\label{sec:gdc}
The objective here is to identify how properties of the target are transferred to the final time adjoints.
To this end, we now focus on a distributed control setting based on the following assumption:
\begin{assumption}\label{ass:glob_contr}
    In this subsection we assume the following:
    \begin{enumerate}
        \item The map~$B$ is a parameter-independent, globally distributed control operator, i.e.~$B$ is the identity on~$\lH$;
        \item The operators~$A_\mu$ are self-adjoint and satisfy uniformly the bounds
        \begin{align}\label{ass:A_unif}
            c_1 A \leq A_\mu \leq c_2 A,
        \end{align}
        for some constants~$c_1,c_2>0$ and a negative definite, self-adjoint operator~$A<0$ with compact resolvent. 
    \end{enumerate}
\end{assumption}
\begin{remark}
    A typical example for a family of operators~$A_\mu$ satisfying the above assumptions is the Laplacian operator~$\nabla\cdot(\mA_\mu (x) \nabla)$ with uniformly bounded coefficient matrices satisfying ~$c_1 \mI \leq \mA_\mu \leq c_2 \mI~$ with constants~$c_1,c_2>0$.
\end{remark}

In order to illustrate how a transport-dominated character in the target states is transferred, let us now make explicit the structure of~$\lA_T$ defined by~\eqref{eq:At} in the distributed control setting. The control system then reads 
\begin{align*}
    \dot{y}_\mu(t) = A_\mu y_\mu(t) + u(t)\ \text{for }t\in[0,T],\qquad y_\mu(0) = y_\mu^0 \in \lH.
\end{align*}
\par
For every~$u\in L^2(0,T ; \lH)$, the state at time~$T$ is given by the Duhamel's formula, that is,
\begin{align*}
    y_\mu(T) = e^{T A_\mu }y_\mu^0 + \int_0^T e^{(T-s) A_\mu} u(s)\d{s}.
\end{align*}
The~\Cref{ass:glob_contr} implies that each ~$A_\mu$ is a self-adjoint operator  with compact resolvent. In particular, it allows for a spectral decomposition and there exists a sequence of corresponding eigenpairs~$\{(\lambda_{\mu,k},\phi_{\mu,k})\}_{k\geq1}$ such that~$\{\phi_{\mu,k}\}_{k\geq 1}$ forms an orthonormal basis in~$\lH$. Moreover, from the assumption~\eqref{ass:A_unif} and  the min-max characterization of eigenvalues, it follows
\begin{align}\label{eq:eig_bounds}
   c_1 \lambda_{k}  \leq \lambda_{\mu,k} \leq c_2 \lambda_{k} ,
\end{align}
where~$(\lambda_{k})_{k\geq 1}$ is the sequence of decreasing eigenvalues of the operator~$A$ (diverging to minus infinity if~$\lH$ is infinite-dimensional).
\par
Moreover, the finite-time controllability Gramian associated with the system and defined by~\eqref{eq:w_gram} is simplified to
\begin{align*}
    W_\mu(T) = \int_0^T e^{(T-s) A_\mu} e^{(T-s) A_\mu}\d{s} = \int_0^T e^{2t A_\mu}\d{t},
\end{align*}
after the change of variable~$t = T-s$. 
It follows immediately from the spectral properties of~$A_\mu$ that~$W_\mu(T)$ is bounded, self-adjoint and compact on~$ \lH$.
More specifically, its representation in the basis consisting of the eigenvectors of~$A_\mu$ is diagonal:
\begin{align*}
    W_\mu(T)\phi_{\mu,k} = \omega_{\mu,k}(T)\phi_{\mu,k},
    \qquad\text{where} \qquad
    \omega_{\mu,k}(T)=\frac{1-e^{2\lambda_{\mu,k} T}}{-2\lambda_{\mu,k}}.
\end{align*}
Consequently,  by  expanding the target and the final time adjoint in this basis as
\begin{align*}
    y_\mu^T = \sum_{k\ge1} y_{\mu,k}^T \phi_k,
    \qquad
    \varphi_\mu^*(T) = \sum_{k\ge1} \varphi_{\mu,k}^*(T)\phi_k,
\end{align*}
the adjoint equation is reduced to the mode-by-mode relation
\begin{align*}
    (1 + \alpha \omega_{\mu,k}(T))\cdot \varphi_{\mu,k}(T)
    =
    - \alpha y_{\mu,k}^T.
\end{align*}
Thus, we obtain
\begin{align*}
    \varphi_{\mu,k}^*(T)
    =
    - a_{\mu,k}(T)\, y_{\mu,k}^T
    \qquad\text{with }
    a_{\mu,k}(T) \coloneqq \frac{\alpha}{1 + \alpha \omega_{\mu,k}(T)}.
\end{align*}
The sequence~$\{a_{\mu,k}(T)\}_{k\ge1}$ is uniformly bounded above and below by positive constants, and therefore defines a bounded isomorphism at the coefficient level.
\par
Indeed, it holds
\begin{align*}
    0\leq \omega_{\mu,k}(T)
    \leq \frac{1}{-2\lambda_{\mu,k}}
    \leq\frac{1}{2 c_1 |\lambda_k|}
    \leq\frac{1}{2 c_1 |\lambda_1|},
\end{align*}
where~$c_1$ is the constant from~\eqref{eq:eig_bounds}.
As a consequence, $\omega_{\mu,k}(T)\to0$ as~$k\to\infty$, uniformly with respect to~$\mu$.  In particular, the multiplier~$a_{\mu,k}(T)$ is bounded from above and below by positive constants independent of~$k$ and~$\mu$. More precisely,
\begin{align*}
    0<a_{\mu,k}(T)=\frac{\alpha}{1+\alpha\omega_{\mu,k}(T)}\leq \alpha,
\end{align*}
and also
\begin{align*}
    a_{\mu,k}(T)=\frac{\alpha}{1+\alpha\omega_{\mu,k}(T)}
    \geq
    \frac{\alpha}{1+\alpha\frac{1}{2c_1|\lambda_1|}}
    \eqqcolon c_T>0.
\end{align*}
This yields upper and lower bounds on the coefficients~$\varphi_{\mu,k}^*(T)$ of the form
\begin{align}\label{eq:comparison_target_adjoint_coefficients}
    c_T |y_{\mu,k}^T|
    \leq
    |\varphi_{\mu,k}^*(T)|
    \leq
    \alpha |y_{\mu,k}^T|
    \qquad\text{for all }
    k\ge1\text{ and }\mu\in\params.
\end{align}

\begin{remark}\label{rm:Sob_reg}
    We highlight that the uniform equivalence given in~\eqref{eq:comparison_target_adjoint_coefficients} implies that the Sobolev-type summability is preserved, that is,
    \begin{align}\label{eq:Fou_pr}
        \sum_{k\geq1}(1+|\lambda_k|)^s |\varphi_{\mu,k}^*(T)|^2
        \lesssim
        \sum_{k\geq1}(1+|\lambda_k|)^s |y_{\mu,k}^T|^2.
    \end{align}
    Thus, every Sobolev-type regularity assumption on the target, formulated in terms of the summability of its coefficients, is inherited by the final time adjoint. Moreover, if the target coefficients satisfy a two-sided decay condition, then the same decay holds for the final time adjoint. 
    \par
    This transfer of regularity from the target to the final time adjoint implies consequences for the approximation properties of the corresponding solution manifold, in particular in the decay of the Kolmogorov widths. Indeed, let us assume that the target coefficients satisfy
    \begin{align*}
        \sum_{k\geq1}(1+|\lambda_k|)^s |y_{\mu,k}^T|^2 \leq C\qquad\text{for all }\mu\in\params
    \end{align*}
    for some constant~$C>0$ independent of~$\mu$. Then, by~\eqref{eq:Fou_pr}, the set of coefficient sequences~$\left\{ (\varphi_{\mu,k}^*(T))_{k\geq 1} \colon\mu\in\mathcal P\right\}$ is contained in a bounded subset of the weighted sequence space~$h_A^s$, where
    \begin{align*}
        h_A^s\coloneqq \left\{ (a_k)_{k\geq 1} \in \ell^2\colon \sum_{k\geq1}(1+|\lambda_k|)^s |a_k|^2<\infty \right\}.
    \end{align*}
    It is well known (see e.g.~\cite[Chapter~VI~\S2]{Pinkus1985}) that the Kolmogorov~$N$-width of the unit ball of~$h_A^s$ in~$\ell^2$ decays polynomially. Consequently, the manifold of final time adjoints admits a polynomial upper bound for its Kolmogorov widths whenever the target coefficients are uniformly bounded in~$h_A^s$.
    \par
    We emphasize, however, that this argument only provides a polynomial upper bound for the Kolmogorov widths. In the next subsection we present an example in which, beyond regularity, the slow Kolmogorov width decay of the target manifold is itself transferred to the manifold of final time adjoints.
\end{remark}

\subsubsection{Moving targets and the transfer of a slow Kolmogorov width decay}\label{ex:TTB}
Here we verify the above analysis of the Kolmogorov width decay on a special example involving moving targets (such as the characteristic functions of a small rectangle whose center depends on~$\mu$). 
In that case, the manifold of final targets~$\lM_T~$ contains translations of a localized profile, which is well-known to lead to slowly decaying Kolmogorov widths~\cite{ohlberger2016reduced,greif2019decay,arbes2025kolmogorov}.
\par
Our objective is to illustrate how a transport-dominated regime in the moving target appears and is transferred to the final time adjoints. We study a particular well-known example, namely the heat equation in a rectangular domain with distributed control. The aforementioned setting is considered as a prototypical test case, since the Dirichlet Laplacian generates an analytic and strongly smoothing semigroup. Hence, if transport-dominated effects persist even in this prototypical and favorable situation, they should be understood as an intrinsic feature of the parameter dependence rather than a consequence of the dynamics or the control mechanism.
\par
In the example we define ~$\Omega=(0,L_1)\times\cdots\times(0,L_d)\subset\mathbb R^d$ and set~$\lH\coloneqq L^2(\Omega)$.
In order to simplify the presentation, we fix the problem setting as follows. 
\begin{enumerate}
    \item The penalization constant~$\alpha$ is equal to~$1$;
    \item The parameter dependence enters only through the moving targets, while the dynamics and the control operator are parameter-independent;
    \item In particular, we consider a globally distributed control operator~$B=I$ and the Dirichlet Laplacian
    \begin{align*}
        A=\Delta,\qquad D(A)=H^2(\Omega)\cap H_0^1(\Omega).
    \end{align*}
\end{enumerate}
The Dirichlet Laplacian~$A=\Delta$ is self-adjoint, negative definite, and has compact resolvent. Its eigenfunctions are given by
\begin{align*}
    \phi_{\mathbf k}(x)
    =
    \prod_{j=1}^d \sqrt{\frac{2}{L_j}}
    \sin\left(\frac{k_j\pi x_j}{L_j}\right)
    \qquad\text{for }
    \mathbf k=(k_1,\dots,k_d)\in\mathbb N^d,
\end{align*}
with corresponding eigenvalues
\begin{align*}
    \lambda_{\mathbf k}
    =
    -\sum_{j=1}^d \left(\frac{k_j\pi}{L_j}\right)^2.
\end{align*}
Thus, for every~$z\in\lH$,
\begin{align*}
    z=\sum_{\mathbf k\in\mathbb N^d} z_{\mathbf k}\phi_{\mathbf k},
    \qquad 
    z_{\mathbf k}\coloneqq\langle z,\phi_{\mathbf k}\rangle_{L^2(\Omega)}.
\end{align*}
Due to the choice of the globally  distributed control, the finite-time (parameter-independent) controllability Gramian~$W(T)$ given by~\eqref{eq:w_gram} is diagonalisable in the basis~${\phi_{\mathbf k}}$ and acts as
\begin{align*}
    W(T)\phi_{\mathbf k}
    =
    w_{\mathbf k}(T)\phi_{\mathbf k},
    \qquad
    w_{\mathbf k}(T)
    =
    \frac{1-e^{2\lambda_{\mathbf k}T}}{-2\lambda_{\mathbf k}}.
\end{align*}
We now introduce a parameter-dependent moving target. Let~$\ell=(\ell_1,\dots,\ell_d)\in(0,\infty)^d$ be the vector of side lengths of a small rectangle~$Q_\ell$. For each parameter~$\mu\in\params\subset\mathbb R^d$, let
\begin{align*}
    c(\mu)=(c_1(\mu),\dots,c_d(\mu))\in\Omega
\end{align*}
denote the center of the rectangle. Each component
\begin{align*}
    c_j(\mu)\colon\params\to (0,L_j)
\end{align*}
represents the position of the center in the~$j$-th spatial direction, and we assume that the mapping~$\mu\mapsto c(\mu)$ is  continuous. Moreover, we assume that the rectangle remains strictly inside~$\Omega$, namely
\begin{align*}
    \frac{\ell_j}{2}<c_j(\mu)<L_j-\frac{\ell_j}{2},
    \qquad\text{for all }j=1,\dots,d\text{ and }\mu\in\params.
\end{align*}
The moving target is defined as the characteristic function of the rectangle, i.e.
\begin{align*}
    y_\mu^T(x)
    =
    \mathbf 1_{Q_\ell(c(\mu))}(x)=\begin{cases}
        1,&\text{if }x\in Q_\ell(c(\mu)),\\
        0,&\text{otherwise},
    \end{cases}
\end{align*}
where
\begin{align*}
    Q_\ell(c(\mu))
    \coloneqq
    \prod_{j=1}^d
    \left(c_j(\mu)-\frac{\ell_j}{2},
          c_j(\mu)+\frac{\ell_j}{2}\right).
\end{align*}
Therefore, its Fourier coefficients can be represented by
\begin{align*}
    (y_\mu^T)_{\mathbf k}
    =
    \prod_{j=1}^d
    \sqrt{\frac{2}{L_j}}
    \frac{2L_j}{k_j\pi}
    \sin\left(\frac{k_j\pi c_j(\mu)}{L_j}\right)
    \sin\left(\frac{k_j\pi \ell_j}{2L_j}\right).
\end{align*}
In particular, the dependence on~$\mu$ is entirely encoded in the factors
\begin{align*}
    \sin\left(\frac{k_j\pi c_j(\mu)}{L_j}\right).
\end{align*}
We now transfer this dependence to the spatial structure by using the identity
\begin{align*}
    \sin a\,\sin b
    =
    \frac12\left[\cos(a-b)-\cos(a+b)\right].
\end{align*}
Applying this relation componentwise yields
\begin{align*}
    \sin\left(\frac{k_j\pi c_j(\mu)}{L_j}\right)
    \sin\left(\frac{k_j\pi x_j}{L_j}\right)
    =
    \frac12
    \left[
    \cos\left(\frac{k_j\pi(x_j-c_j(\mu))}{L_j}\right)
    -
    \cos\left(\frac{k_j\pi(x_j+c_j(\mu))}{L_j}\right)
    \right].
\end{align*}
Consequently, there holds
\begin{align*}
    (y_\mu^T)_{\mathbf k}\phi_{\mathbf k}(x)
    =
    \prod_{j=1}^d
    \frac{2}{k_j\pi}
    \sin\left(\frac{k_j\pi \ell_j}{2L_j}\right)
    \left[
    \cos\left(\frac{k_j\pi(x_j-c_j(\mu))}{L_j}\right)
    -
    \cos\left(\frac{k_j\pi(x_j+c_j(\mu))}{L_j}\right)
    \right].
\end{align*}
As the next step we examine how this parameter dependence is transferred to the final time adjoints. According to~\eqref{eq:final_adjoint_system} and taking into account~$\alpha=1$, these satisfy the relation
\begin{align*}
    (I+W (T))\varphi_\mu^*(T)
    =
    -y_\mu^T.
\end{align*}
As the Gramian acts diagonally we obtain that
\begin{align}\label{eq:rep_shift}
    \varphi_\mu^*(T,x)
    =
    -\sum_{\mathbf k\in\mathbb N^d}
    \frac{(y_\mu^T)_{\mathbf k}}{1+w_{\mathbf k}(T)}
    \phi_{\mathbf k}(x).
\end{align}
Therefore, the Gramian acts diagonally and preserves the parametric structure at the level of coefficients.
The Fourier coefficients of the terminal adjoints have the same form as those of the prescribed final targets, up to a multiplication by a scalar~$1/(1+w_{\mathbf k}(T))$. 
Hence, the terminal adjoint is a regularized version of the indicator function of the rectangle~$Q_\ell$, whose center~$c(\mu)$ moves with the parameter. This regularization does not affect the Sobolev-regularity, which is preserved by the coefficient equivalence established in~\Cref{rm:Sob_reg}.
Consequently, the solution set~$\{\varphi_\mu^*(T):\mu\in\params\}$ consists of superpositions of shifted profiles, and therefore retains the transport-dominated structure of the target.
\begin{remark}
    The representation~\eqref{eq:rep_shift} should not be interpreted as stating that the final time adjoints are exactly translations of a single profile.
    Such a property is not expected in bounded domains with homogeneous Dirichlet boundary conditions.
    Nevertheless, the parameter dependence induced by the moving target is still retained in the Fourier representation, showing the transfer of the transport-dominated structure.
\end{remark}
Summarizing, the manifold~$\mathcal{M}_\varphi$ consisting of all final time adjoints inherits the transport-dominated structure of the target. In particular, for discontinuous targets moving depending on the parameter, the Kolmogorov widths decay slowly, which implies that linear subspace-based reduced-order models are inefficient for approximating the final time adjoints.

\subsection{Implications for linear reduced models}
In this section we have studied the Kolmogorov widths decay of the solution manifold in some special cases (assuming parameter-independent dynamics and/or globally distributed control operators). The considered examples reveal some of the fundamental limitations of linear projection-based reduced-order models. Even in the diagonal, distributed control setting, where the Gramian has the simplest possible action, the parameter-dependent moving target generates a transport-dominated family of terminal adjoints. Similarly, the regularity of targets is transferred to the final time adjoints. To the best of our knowledge, this influence of the parameter on the final time adjoint was not known beforehand. We emphasize that this phenomenon persists even for the heat equation: although the associated analytic semigroup provides strong regularization, it is not sufficient to overcome the difficulties induced by the transport-dominated component of a moving target. Thus, the slow decay of the associated Kolmogorov widths should be understood as a structural feature of the parameter dependence, rather than as a complication produced by the control operator and the underlying dynamics.
\par
This situation naturally appears when the parameterization induces spatial shifts of localized structures. As shown by the moving-target discussion in~\Cref{ex:TTB}, such manifolds cannot be approximated in an efficient manner by low-dimensional linear spaces. As a result, linear projection-based reduced-order models require increasingly large dimensions to achieve a reasonable accuracy. The computational costs of the online phase for the linear reduced model from~\Cref{sec:linear-reduced-model} scales linearly with the reduced dimension (see also the numerical experiments in~\Cref{sec:numerical-experiments}), which makes the online phase costly and reduces the efficiency of the method when a large reduced basis is required. This motivates us to seek for alternative approaches, for instance based on machine learning tools, as the one presented in the next section.

\section{A nonlinear machine learning-based surrogate}
\label{sec:nonlinear-surrogate}
In this section we describe a surrogate model for solving parametric optimal control problems of the form~\eqref{eq:PDOC_2}.
Our objective is to approximate the optimal final time adjoint determined by~\eqref{eq:final_adjoint_system}, since the optimal control, the adjoint trajectory, and the state trajectory are all uniquely determined by it through the optimality system~\eqref{eq:opt_sys_a}.
Moreover, approximating the final time adjoint preserves the residual-based a posteriori error estimator~\eqref{eq:error-estimator} which can be evaluated for any candidate approximation and therefore remains applicable beyond linear reduced spaces.
For these reasons, we focus on approximating the optimal final time adjoint rather than learning the control directly as a function of the parameter.
\par
In the next subsection, we introduce the neural network architecture specifically designed to approximate the solutions to the final time adjoint equation~\eqref{eq:final_adjoint_system}.

\subsection{U-Net architecture mapping parametric fields to final time adjoints}\label{sec:u-net-architecture}
The design of the proposed surrogate is guided by the theoretical results presented in~\Cref{sec:LROM_lim}.
As shown in~\Cref{sec:transport-dominated-effect-moving-target}, transferring information about the target positions to the final time adjoint turns out to be essential.
Consequently, accurate approximations of the solution manifold require preserving information about the location and geometry of the relevant spatial features. The neural network architecture employed in this work is designed precisely for this purpose.
Specifically, we employ a~U-Net architecture~\cite{ronneberger2015unet} that has been proven particularly effective for imaging tasks that require incorporating the surrounding context while preserving spatial information. Interpreting parameter-dependent fields arising from~PDEs as images, the~U-Net architecture is therefore particularly well suited to our application. Related approaches for parametric~PDEs have recently been discussed in~\cite{heiss2023multilevel,schuette2024multilevel}.
\par
Since for the problems we are interested in, spatial information such as the position of the support of the target state plays a crucial role, we incorporate this information directly into the input of the surrogate.
Instead of passing the parameter itself to the surrogate, we provide the full parametric fields, such as for instance, the initial condition, target state, and diffusion coefficient, as input to the neural network.
The output of the neural network will be an approximation of the corresponding final time adjoint.
\par
In other words, rather than learning the mapping~$\mu\mapsto\varphi_{\mu}^*(T)$ directly, we approximate the solution operator associated with the final time adjoint equation~\eqref{eq:final_adjoint_system}, i.e.
\begin{align*}
    (I+\alpha W_\mu(T))\varphi_\mu^*(T)
    =
    \alpha\left(T e^{A_\mu}y_\mu^0-y_\mu^T\right).
\end{align*}
To this end, let~$\mathcal X = \mathcal Y_0\times \mathcal Y_T \times \mathcal K \times\cdots$ denote the collection of parameter-dependent fields that determine the controllability Gramian~$W_\mu(T)$ and the right-hand side of~\eqref{eq:final_adjoint_system}, where~$\mathcal Y_0$, $\mathcal Y_T$, and~$\mathcal K$ represent suitable spaces for the initial condition, target state, diffusion coefficient, and other parameter-dependent quantities. We then train a neural network~$\Phi_\mathrm{U-Net}\colon\mathcal X\to\lH$ to approximate the mapping
\begin{align*}
    (y_\mu^0,y_\mu^T,\kappa_\mu,\ldots) \longmapsto \varphi_\mu^*(T).
\end{align*}
In contrast to approaches that learn a parameter-to-solution map, the input consists of fields that explicitly encode the information defining the operator and the right-hand side of the final time adjoint equation.
\par
The~U-Net architecture itself consists of a downsampling and an upsampling part in which the spatial resolution is reduced while adding more channels in the intermediate layers of the network.
In particular, the~U-Net does not perform a compression of the input data in the sense of a reduced-order model.
The input is instead transformed through hidden layers that typically contain even more dimensions than the input or the output themselves.
For instance in our numerical experiment in~\Cref{sec:ball-example}, the input dimension is~$8192$, the output size is~$4096$, whereas the bottleneck (see also~\Cref{fig:u-net-architecture}) is of size~$16384$ (with a spatial reduction from~$64\times64$ to~$8\times8$).
As we will see in our numerical experiments, the approach nevertheless leads to significant speedups compared to the full-order model.
The~U-Net further comprises skip-connections that pass feature maps of layers from the downsampling path to corresponding layers at the same spatial resolution in the upsampling part by channel concatenation (without any transformation).
The residual connection from the input to the output layer allows the output to be of the same regularity as the input.
As shown in the previous section, the regularity is preserved when mapping target positions to optimal final time adjoints. Consequently, preserving the regularity of the input is a desirable property of U-Net architectures, which is generally difficult to achieve using purely convolutional networks~\cite{habring2023note}.
\par
We point further out that when only the parameter vector is provided as input, the spatial distribution of the parametric fields is implicitly encoded and must first be reconstructed by the surrogate.
In contrast, using the parameter-dependent fields allows the network to access this information explicitly throughout the forward pass and the learning process.
Moreover, the~U-Net architecture we consider is based on convolutional and transposed-convolutional layers.
These layers are translation-equivariant, which in particular implies that the network can transfer spatial behavior observed in the training data to other locations in the spatial domain.
\par
An illustration of the architecture with inputs and outputs is provided in~\Cref{fig:u-net-architecture}.
The architecture we use consists of a downsampling part of three levels with two consecutive~$3\times3$-convolution layers (the number of channels doubles between levels), each convolutional layer followed by the~Gaussian Error Linear Unit~(GELU) activation function (defined as~$\rho(x)=x\cdot\Phi(x)$ where~$\Phi\colon\R\to[0,1]$ is the cumulative distribution function of the Gaussian distribution), and after each level a max-pooling with stride~$2$.
The bottleneck part of the~U-Net (at the very bottom in~\Cref{fig:u-net-architecture}) is made of two~$3\times3$-convolutions with~GELU activation and contains in total eight times as many channels as the first layer.
The final number of channels in the first level is set to~$32$, leading to~$256$ channels in the bottleneck path.
The upsampling part of the network is constructed by three levels of transposed convolutions with stride~$2$ (which halves the channel count and doubles the spatial extent), then channel-wise concatenation with the corresponding encoder skip-connection (which uses the features from the downsampling path in the respective levels \emph{before} applying the max-pooling), and finally two~$3\times 3$-convolutions, each followed by a~GELU activation.
Finally, an additional~$1\times 1$-convolutional layer is used to reduce the result after the upsampling stage to a single channel output.
\begin{figure}[htbp]
    \centering
    \resizebox{\textwidth}{!}{
        \begin{tikzpicture}
            \node[yslant=1] at (-2,0,0) {\includegraphics[width=2cm, height=4cm]{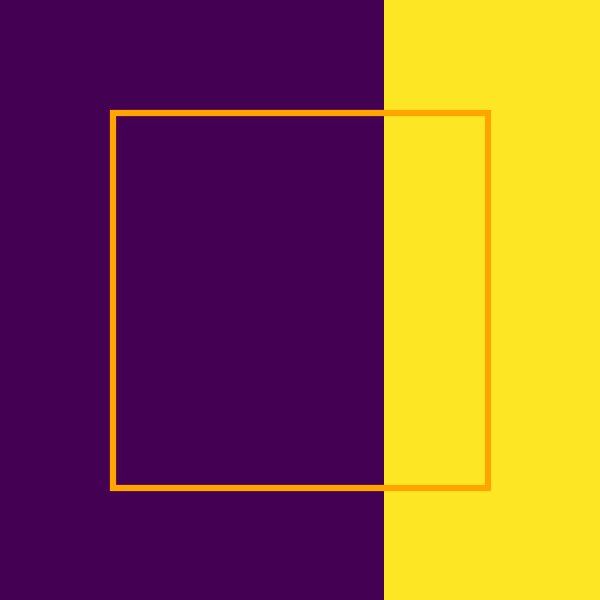}};
            \node[yslant=1] at (-1,0,0) {\includegraphics[width=2cm, height=4cm]{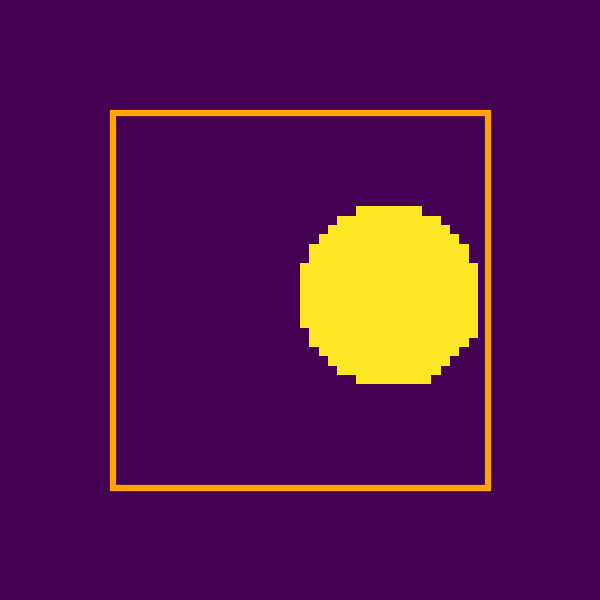}};
            \node at (-3,-3.75,-1) {\LARGE$\kappa_\mu$};
            \node at (-2,-3.75,-1) {\LARGE$y_\mu^T$};
            
            \cnnlayerend{0.0}{0}{0}{0.1}{4}{4}{orange};
            \cnnlayerend{0.5}{0}{0}{0.1}{4}{4}{orange};
            
            \cnnarrow{0.55}{0}{0}{1.95}{-3}{0};
            \cnnarrow{1.55}{0}{0}{18.55}{0}{0};
            
            \cnnlayerend{2.0}{-3}{0}{0.1}{3}{3}{red};
            \cnnlayerend{2.55}{-3}{0}{0.2}{3}{3}{orange};
            \cnnlayerend{3.15}{-3}{0}{0.2}{3}{3}{orange};
            
            \cnnarrow{3.25}{-3}{0}{4.6}{-6}{0};
            \cnnarrow{4.25}{-3}{0}{14.4}{-3}{0};
            
            \cnnlayerend{4.7}{-6}{0}{0.2}{1.5}{1.5}{red};
            \cnnlayerend{5.4}{-6}{0}{0.4}{1.5}{1.5}{orange};
            \cnnlayerend{6.2}{-6}{0}{0.4}{1.5}{1.5}{orange};
            
            \cnnarrow{6.4}{-6}{0}{7.6}{-9}{0};
            \cnnarrow{7.4}{-6}{0}{10.8}{-6}{0};
            
            \cnnlayerend{7.8}{-9}{0}{0.4}{1}{1}{red};
            \cnnlayerend{8.8}{-9}{0}{0.8}{1}{1}{orange};
            \cnnlayerend{10.0}{-9}{0}{0.8}{1}{1}{orange};
            
            \cnnarrow{10.4}{-9}{0}{11.8}{-6}{0};
            
            \cnnlayer   {12.0}{-6}{0}{0.4}{1.5}{1.5}{green};
            \cnnlayerend{12.4}{-6}{0}{0.4}{1.5}{1.5}{green};
            \cnnlayerend{13.2}{-6}{0}{0.4}{1.5}{1.5}{orange};
            \cnnlayerend{14.0}{-6}{0}{0.4}{1.5}{1.5}{orange};
            
            \cnnarrow{14.2}{-6}{0}{15.4}{-3}{0};
            
            \cnnlayer   {15.5}{-3}{0}{0.2}{3}{3}{green};
            \cnnlayerend{15.7}{-3}{0}{0.2}{3}{3}{green};
            \cnnlayerend{16.3}{-3}{0}{0.2}{3}{3}{orange};
            \cnnlayerend{16.9}{-3}{0}{0.2}{3}{3}{orange};
            
            \cnnarrow{17.0}{-3}{0}{19.55}{0}{0};
            
            \cnnlayer   {19.6}{0}{0}{0.1}{4}{4}{green};
            \cnnlayerend{19.7}{0}{0}{0.1}{4}{4}{green};
            \cnnlayerend{20.2}{0}{0}{0.1}{4}{4}{orange};
            \cnnlayerend{20.7}{0}{0}{0.1}{4}{4}{orange};
            
            \cnnarrow{20.75}{0}{0}{22.95}{0}{0};
            
            \cnnlayerend{23.0}{0}{0}{0.1}{4}{4}{blue};
            
            \node[yslant=1] at (24,0,0) {\includegraphics[width=2cm, height=4cm]{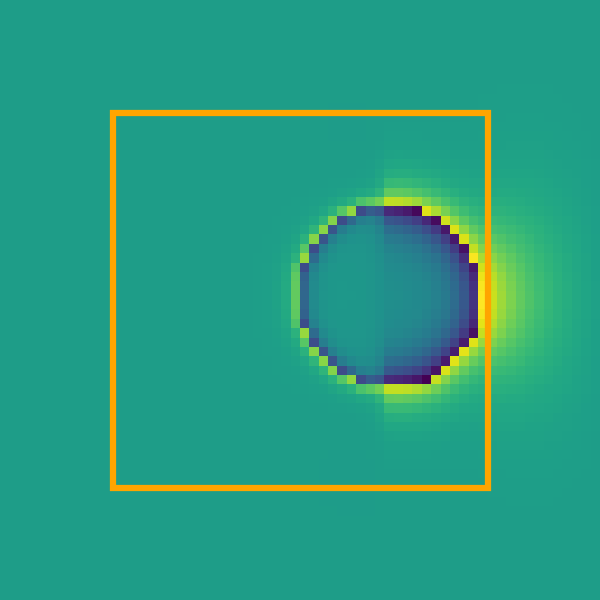}};
            \node at (23.5,-3.75,-1) {\LARGE$\varphi_\mu^*(T)$};
        \end{tikzpicture}
    }
    \caption{U-Net architecture with parametric fields as inputs and optimal final time adjoints as output. The orange boxes represent convolutional layers followed by~GELU activation functions and the red boxes refer to max-pooling layers. The green boxes depict transposed convolutions together with the channels concatenated from the corresponding encoder level via a skip connection. The blue box corresponds to the final~$1\times1$-convolution that maps the feature stack to the single output channel.}
    \label{fig:u-net-architecture}
\end{figure}
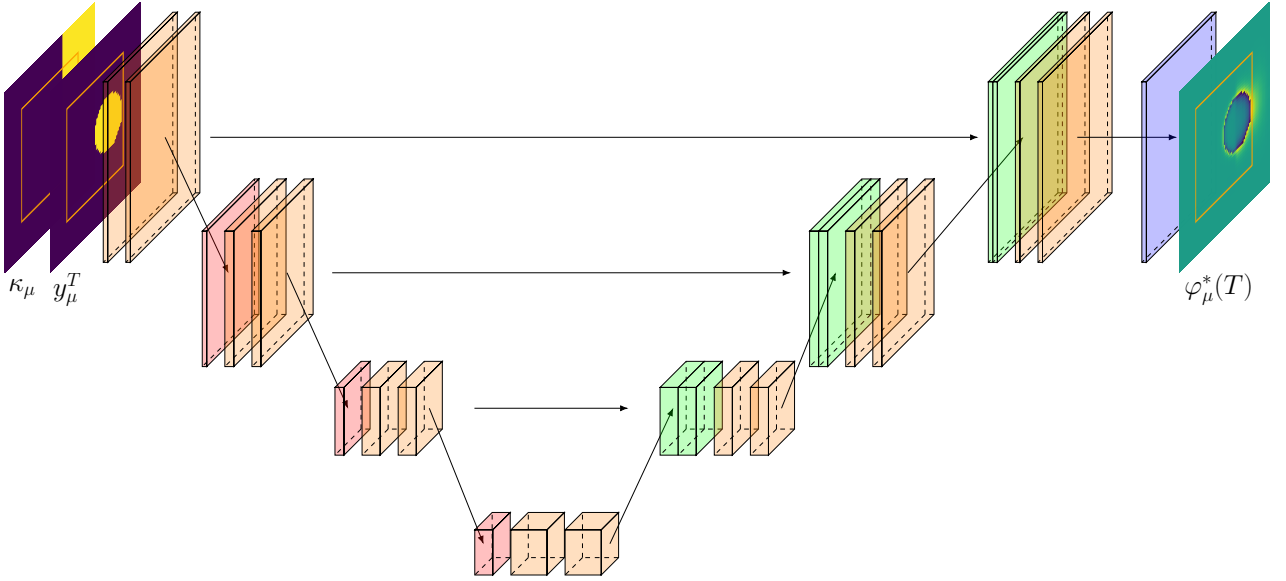

\subsection{Discussion on the nonlinear approach}\label{sec:discussion-nonlinear-approach}
We conclude this section clarifying the main advantages and limitations of the proposed nonlinear framework and compare it with existing approaches for parametrized optimal control problems.

\paragraph{Advantages of the method.}
The~U-Net combines several useful properties for the transport-dominated problems considered here.
By operating directly on parameter-dependent fields rather than low-dimensional parameter vectors, it preserves the spatial information associated with localized and translated structures that characterize the transport-dominated solution manifolds analyzed in~\Cref{sec:LROM_lim}.
Moreover, the online evaluation consists only of a neural network forward pass followed by the reconstruction of the corresponding state and control.
Finally, since the surrogate approximates the final time adjoint, the residual-based a posteriori error allows the quality of the prediction to be assessed after the approximation has been computed.
\par
We further remark that the~U-Net approach imposes a structural assumption on the way that parameters enter the optimal control problem. 
Rather than taking the parameter~$\mu$ itself as input, the~U-Net receives a finite collection of parameter-dependent fields represented on a computational grid.
Accordingly, we assume that all parameter-dependent quantities determining operators~$A_\mu$ and~$B_\mu$ can be represented through such fields, which can be naturally identified with elements of the discretized state space.
This setting covers the standard situations considered in optimal control problems, such as parameter-dependent coefficient fields, source terms, initial conditions, and target states.
\paragraph{Limitations of the method.}
This approach nevertheless has limitations.
As with any data-driven surrogate, an offline training phase is required and its quality depends on the available training set.
The method does not eliminate the offline cost associated with the generation of training data, since each training sample requires the solution of a full optimal control problem.
As we will see though, the~U-Net needs much less training data than, for instance, an autoencoder-based approach to reach similar errors.
In addition, convolutional architectures assume that the computational domain can be represented as a rectangle and the discretization is given as a regular grid.
Extensions to more general domains may require more specialized approaches.
\par
The limitation also arises when the parameter changes the problem in a way that cannot be encoded by a finite collection of fields on the computational grid.
For instance, our formulation does not directly incorporate parameter-dependent boundary conditions or more general operator-valued parameterizations requiring the discretized operator itself as an input to the network.
Moreover, the number of parameter-dependent fields is assumed to remain relatively small, similarly to an affine parameterization in reduced basis methods~\cite{ohlberger2016reduced}.
Potential extensions to more general parameterizations and discretizations, for instance parameter-dependent boundary conditions, will be discussed in the outlook of the paper.
\paragraph{Comparison with existing approaches.}
The~U-Net approach should be viewed as complementary to the existing reduced-order approaches.
As is well known, linear reduced-order models are effective whenever the solution manifold admits a rapidly decaying Kolmogorov width.
However, \Cref{sec:LROM_lim} shows that this assumption fails in transport-dominated situations.
Consequently, increasing the dimension of the reduced space may lead only to limited improvements in approximation quality.
\par
Compared with parameter-to-solution approaches (such as parametric decoders, see~\Cref{sec:parametric-decoders}) or with parameter-to-solution operator approaches (DeepONets~\cite{lu2021learning}, or hypernetwork-based models~\cite{chauhan2024brief}), the present framework incorporates the relevant spatial information explicitly through the input fields.
This avoids the need to reconstruct geometric information solely from a low-dimensional parameter vector and is particularly useful in the small-data regime.
\par
Another related family of approaches employs autoencoders in combination with a parameter-to-latent mapping~\cite{fresca2021comprehensive,franco2023deep}.
There the solution is first compressed into a latent representation and subsequently reconstructed, see~\Cref{sec:autoencoders-param-to-latent} for details.
While such methods have proven successful for many parametrized~PDEs, transport-dominated structures remain challenging when only a limited number of training samples is available.
In particular, the spatial information associated with translated or localized features must be compressed into a low-dimensional latent space before being recovered by the decoder.
The proposed framework avoids this intermediate compression step by learning the mapping from parameter-dependent fields directly to the corresponding final time adjoints.
The numerical experiments of~\Cref{sec:numerical-experiments} compare these approaches and illustrate the advantages and limitations of each methodology.

\section{Numerical experiments}\label{sec:numerical-experiments}
In order to evaluate the practical performance of the proposed approach, we focus on two numerical examples: The first example in~\Cref{ssec:HGC} is similar to the one we examined theoretically in~\Cref{sec:gdc}.
The second test case in~\Cref{sec:ball-example} is more involved as it considers a localized control region, a moving ball as target state and a discontinuous diffusivity field where the location of the discontinuity changes with the parameter as well.
\par
Before discussing the numerical experiments in detail, we first introduce the methods against which we compare our approach in~\Cref{ssec:CM} and then provide implementation details in~\Cref{sec:implementation-details}.

\subsection{Methods considered for comparison}\label{ssec:CM}
In the following subsections, we summarize several alternative methods presented in the literature to which we compare our~U-Net based approach.
These include linear reduced models as discussed above as well as nonlinear methods based on autoencoder architectures.
Here, we aim to briefly discuss the gains and difficulties present in each approach.
This comparison allows us to assess separately the influence of reduced-order modeling and nonlinear representation versus the framework given by the~U-Net on the resulting approximation quality and computational performance.

\subsubsection{Linear reduced model and basis generation}
The first method to which we compare our approach is the reduced basis method presented in~\Cref{sec:linear-reduced-model}.
In order to select the reduced basis we consider a (weak) greedy procedure as described in detail in~\cite{lazar2016greedy}.
The basis vectors chosen by the greedy algorithm coincide with solution snapshots (i.e.~final time adjoints) selected according to an a posteriori error estimator.
The greedy approach provides approximation rates comparable to optimal ones, expressed in terms of Kolmogorov widths.
We do not describe the method in detail here; rather, we refer the interested reader to~\cite{JL21}.
In the following, the reduced basis model constructed by the greedy algorithm is referred to as~RB-ROM.
\par
At this point, we also remark that other algorithms to compute a reduced basis, such as the proper orthogonal decomposition~\cite{sirovich1987turbulence}, can be applied as well.
However, the reduced-order model is limited by the linearity of the ansatz space, and this limitation is independent of the method used to choose the reduced space.
We therefore focus here on the greedy method and only present the results obtained by that approach.
\par
The performance of the~RB-ROM is closely related to the approximation properties of the manifold of final time adjoints.
In particular, when the manifold's Kolmogorov widths decay slowly, a large number of basis functions is required to achieve a prescribed accuracy.
\par
In addition to the reduced basis approximation obtained from solving the reduced version of the optimality system~\eqref{eq:opt_sys_a}, we also consider the orthogonal projection of the exact final time adjoint onto the reduced space.
More precisely, given a reduced space~$\lH_N$, we compute the orthogonal projection~$P_{\lH_N}\varphi_{\mu}^*(T) \in \lH_N$ which is the unique element minimizing the approximation error~$\lVert\varphi_\mu^*(T) - v\rVert$ for~$v\in \lH_N$.
Consequently, this projection represents the best approximation that can be achieved within the chosen reduced space. 
\par
We emphasize that the latter approach is neither efficient nor intended for practical computations, as it requires knowledge of the solution itself, namely the optimal final time adjoint~$\varphi_{\mu}^*(T)$.
In this respect, it differs fundamentally from the reduced basis techniques discussed above, where the approximation is obtained by solving a small linear system in the reduced space, without need for the high-fidelity solution.
We present this approach solely to illustrate the approximation capability of the reduced space and to separate the error introduced by the subspace from the additional error coming from the reduced modeling.

\subsubsection{Autoencoder and parameter-to-latent map}\label{sec:autoencoders-param-to-latent}
The authors in~\cite{franco2023deep,fresca2021comprehensive} proposed to train an autoencoder on a set of solution snapshots that aims to approximate the identity map on the training data.
While doing so, the data is compressed into a low-dimensional representation using the encoding part~$\Phi_\mathrm{enc}\colon\R^n\to\R^N$ of the autoencoder with latent dimension~$N\ll n=\dim\lH$.
The decoder part~$\Phi_\mathrm{dec}\colon\R^N\to\R^n$ then reconstructs an approximation of the input data from the latent code produced by the encoder.
In the case of spatial data, convolutional layers are typically used to take spatial relations into account for encoding and decoding. 
\par
For a parameter~$\mu\in\params$, the autoencoder thus approximates the final time adjoint as
\begin{align}\label{eq:AE_approx}
    \varphi_\mu^*(T) \approx \Phi_\mathrm{dec}(\Phi_\mathrm{enc}(\varphi_\mu^*(T))).
\end{align}
Of course, this approach is not directly feasible, as this approximation requires knowledge of the final time adjoint itself.
Therefore, an additional neural network~$\Phi_\mathrm{pred}\colon\params\to\R^N$ is trained that maps a parameter to the corresponding latent representation obtained by passing the true solution through the encoder.
During the online phase, a new parameter is passed through the small neural network producing the latent code.
The decoder is then applied to the latent representation to reconstruct the approximate solution in the high-dimensional solution space.
For a parameter~$\mu\in\params$, the approximate final time adjoint is therefore given as~$\Phi_\mathrm{dec}(\Phi_\mathrm{pred}(\mu))\in\R^n$.
Note that here as an input we have just a parameter value, unlike in~\eqref{eq:AE_approx}, where the input was (an unknown) final time adjoint.
\par
The autoencoder can be interpreted as a nonlinear method for data compression or dimensionality reduction, whereas the approach presented in the previous subsection constitutes its linear counterpart.
In the numerical examples  we refer to this approach as~CNN-AE-ROM.
\par
We highlight that the autoencoders and parameter-to-latent approach is oriented to beat the linear reduced model approach in the sense that the decoder parametrizes a nonlinear approximation manifold rather than a fixed linear subspace.
In particular, this is relevant for solution sets that are known to be poorly approximated by low-dimensional linear spaces.
In these cases, a nonlinear latent representation can in principle encode the dominant geometric variability with fewer degrees of freedom than a reduced basis.
\par
Nevertheless, this additional flexibility does not completely remove the difficulty in the transport-dominated setting considered in this manuscript.
The autoencoder learns a nonlinear representation of the final time adjoint manifold, but the online approximation is still produced from a finite-dimensional parameter-to-latent map.
In particular, it does not explicitly use the spatial structure of the target, the propagated initial datum, or the parameter-dependent coefficients as input fields.
This limits its ability to accurately resolve the sharp spatial displacements and localized features that arise in the examples below.
\par
Similar to considering the orthogonal projection onto the reduced space in the previous section, we also compute an upper bound for the error that can be achieved with the~CNN-AE-ROM by performing a forward pass through the autoencoder.
So instead of using the parameter-to-latent map, we compute the latent representation via passing the true final time adjoint to the encoder.
In contrast to the orthogonal projection this only provides an upper bound for the error reachable by the decoder.
It nevertheless allows us to distinguish between the autoencoder and the parameter-to-latent map as the main source of error.
We refer to the results obtained by a forward pass of the exact final time adjoint through the autoencoder as~CNN-AE.
Since it requires the optimal final time adjoint, this method is again not suitable for practical use and only serves for comparison.

\subsubsection{Parametric decoder}\label{sec:parametric-decoders}
The parametric decoder approach is a variant of the autoencoder method described in the previous subsection.
Instead of training a full autoencoder and learning a parameter-to-latent map, we here only learn a decoder~$\Phi_\mathrm{p-dec}\colon\params\to\R^n$ (referred to as ``parametric decoder'' below) that takes the parameter as input and approximates the mapping
\begin{align*}
    \mu\longmapsto\varphi_\mu^*(T).
\end{align*}
The architecture of the parametric decoder is chosen exactly the same as for the decoder part of the autoencoder above (using~$N=\dim(\params)$ as input dimension), namely as one fully-connected layer followed by transposed convolutional layers.
In contrast to the previous approach, no encoder or latent space is introduced during training.
\par
This architecture serves as a natural point of comparison with the autoencoder-based method.
Since both approaches employ the same decoder architecture, the comparison isolates the role of the learned nonlinear latent representation.
The parametric decoder avoids the additional optimization problem of constructing a suitable latent representation and can therefore be viewed as a compressed version of the autoencoder approach that merges the parameter-to-latent map and the decoder into a single network.
As the numerical experiments demonstrate, however, separating these two tasks and allowing for a latent representation different from the parameter itself leads to slightly better approximation accuracy in practice.

\subsection{Implementation details}\label{sec:implementation-details}
In the following section we provide several details regarding the numerical experiments and the implementation of the previously described methods.
\paragraph{Discretization of the FOM.}
In both examples we consider as spatial domain the unit square~$\Omega=(0,1)^2$ and discretize the respective heat equation in space with a~5-point finite difference Laplacian on a uniform~$64\times 64$~grid leading to~$n=4096$ inner degrees of freedom.
The diffusivity interface in the second example is handled via harmonic averaging on cell faces.
For temporal discretization we apply~$n_t=400$ steps of the~L-stable~BDF-2 time stepper with a single backward Euler step for initialization of the~BDF-2 solver.
We precompute the~LU-decompositions of the respective matrices for time stepping and reuse them whenever possible.
The~FOM linear system for the final time adjoint is solved using the conjugate gradients method~(CG)~\cite{saad2003iterative} with a relative residual tolerance of~$10^{-8}$ and a maximum of~$500$ iterations.
\paragraph{Training and test parameters.}
For the construction of the different surrogate models described above, we consider~$n_\mathrm{train}\in\{100,200,500,1000\}$ independent and uniformly drawn training parameters.
Further, another~$100$ parameters are used as validation parameters for the neural network training.
To make the comparison fair, the validation parameters are also used in the greedy selection, such that the greedy is based on the same~$1100$ parameters as the machine learning-based approaches (it does not require validation parameters).
Another~$100$ independently and uniformly chosen parameters are employed during testing.
The parameter sets (training and test) are the same for all considered methods.
\paragraph{Error metrics.}
We consider the following three error metrics measuring average errors in the final time adjoints, terminal state and control over the test parameter set~$\params_\mathrm{test}\subset\params$, where~$\hat{\varphi}_\mu(T)$, $\hat{y}_\mu(T)$ and~$\hat{u}_\mu$ denote approximate final time adjoint, terminal state and control of the considered surrogate:
\begin{align*}
    \bar{\varepsilon}_\varphi &\coloneqq \frac{1}{|\params_{\mathrm{test}}|}\sum\limits_{\mu\in\params_{\mathrm{test}}}\frac{\lVert\varphi_\mu^*(T)-\hat{\varphi}_\mu(T)\rVert_{L^2(\Omega)}}{\lVert\varphi_\mu^*(T)\rVert_{L^2(\Omega)}}, \\
    \bar{\varepsilon}_y &\coloneqq \frac{1}{|\params_{\mathrm{test}}|}\sum\limits_{\mu\in\params_{\mathrm{test}}}\lVert y_\mu^*(T)-\hat{y}_\mu(T)\rVert_{L^2(\Omega)}, \\
    \bar{\varepsilon}_u &\coloneqq \frac{1}{|\params_{\mathrm{test}}|}\sum\limits_{\mu\in\params_{\mathrm{test}}}\frac{\lVert u_\mu^*-\hat{u}_\mu\rVert_{L^2([0,T],\mathcal{U})}}{\lVert u_\mu^*\rVert_{L^2([0,T],\mathcal{U})}}.
\end{align*}
We remark at this point that~$\bar{\varepsilon}_y$ is reported using the absolute~$L^2$-norm in space because the true target states~$y_\mu^T$ only differ by translations and are therefore independent of~$\mu$.
\paragraph{Neural network training and architectures.}
For the neural network training, Adam~\cite{kingma2015adam} was used as optimizer with an initial learning rate of~$10^{-3}$.
We further minimize the mean squared error loss in the final time adjoint and the latent representation, respectively.
For all neural networks, we perform early-stopping based on the validation loss with a patience of~$1000$ epochs.
Due to early-stopping, the training of all methods stopped well before reaching the maximum number of training iterations ($3000$ for autoencoders and the parameter-to-latent regressors, $10000$ for the~U-Nets and the parametric decoder).
In neural network training, we perform three restarts using random initializations of weights and biases.
Further, we consider three random splits of the available data into training and validation data.
The numbers presented below refer to the best results over the nine training runs chosen according to the smallest validation loss in the parameter-to-latent regressor.
Our numerical experiments showed that the autoencoder performance was less influenced by the latent dimension than the parameter-to-latent map, which is why we used the validation loss of the latter to determine the latent dimension.
\par
For the parametric decoder, we consider a fully-connected layer mapping from the parameter space to~$\R^{128\cdot4\cdot4}$ followed by four~$4\times4$~transposed-convolutions with input channels~$128$, $64$, $32$, $16$ and output channels~$64$, $32$, $16$, $1$, stride~$2$, padding~$1$, and~GELU activations between all but the last layer.
Finally, another transposed-convolutional layer producing a single output layer is used.
\par
For the convolutional autoencoder, the following architecture is used: The encoder consists of four~$3\times 3$~convolutional layers with channels~$1$, $16$, $32$, $64$ and~$128$, stride~$2$, padding~$1$ and~GELU activations.
The bottleneck therefore has a spatial size of~$4\times 4$ and therefore a flattened dimension of~$128\cdot4\cdot4 = 2048$, which is mapped with a fully-connected layer to~$\R^N$ with latent dimension~$N\in\{8,16,32,64\}$.
The decoder mirrors the encoder using~$4\times 4$ transposed convolutions with stride~$2$.
The latent regressor is a multi-layer perceptron mapping from~$\params$ to~$\R^N$ with two~GELU hidden layers of width~$256$ each.
\par
The~U-Net architecture is described in detail in~\Cref{sec:u-net-architecture} and shown in~\Cref{fig:u-net-architecture}.
\paragraph{Hardware and software used for the experiments.}
All experiments were performed on a system running Ubuntu~24.04.4, equipped with an~AMD~EPYC~9654~96-Core Processor with~2.2~TB of~RAM.
Neural network training was additionally performed using an~NVIDIA~H200~GPU with~141~GB of~HBM3e memory on the same machine\footnote{We remark at this point that all experiments can also be performed on smaller systems with less computational power.}.
\par
All algorithms are implemented in the \texttt{Python} programming language using~\texttt{numpy}~\cite{harris2020array} and~\texttt{scipy}~\cite{virtanen2020SciPy} for efficient discretization of the~PDEs and optimal control systems as well as~\texttt{PyTorch}~\cite{paszke2019pytorch} for neural network training.
The sourcecode used to derive the results reported in the next sections is provided in~\cite{sourcecode}.

\subsection{Globally distributed control with moving targets}\label{ssec:HGC}
Our first numerical example is motivated by the theoretical investigations from~\Cref{sec:transport-dominated-effect-moving-target}.
\par
Let~$\Omega=(0,1)^2$, $T=1$, and consider the parameter domain
\begin{align*}
    \params =[0.25,0.75]^2 \subset\R^2.
\end{align*}
For each parameter value~$\mu=(\mu_1,\mu_2)\in\params$, we consider the optimal control problem governed by the heat equation
\begin{equation*}
    \left\{
    \begin{aligned}
        \partial_t y_\mu-\Delta y_\mu &= u_\mu && \text{in }(0,T)\times\Omega,\\
        y_\mu &= 0 && \text{on }(0,T)\times\partial\Omega, \\
        y_\mu(0,x,y) &= 0 && \text{for }(x,y)\in\Omega.
    \end{aligned}
    \right.
\end{equation*}
The control acts on the whole domain, that is, $u_\mu\in L^2(0,T;\Omega)$.
\par
The objective is to steer the state towards a parameter-dependent target state~$y_\mu^T$ by minimizing
\begin{align*}
    J_\mu(u) = \frac{\alpha}{2} \lVert y_\mu(T)-y_\mu^T\rVert_{L^2(\Omega)}^2 + \frac{1}{2}\lVert u\rVert_{L^2(0,T;\Omega)}^2
\end{align*}
for~$\alpha=1000$. To generate a transport-dominated target manifold, we consider translations of a discontinuous rectangular profile. Let~$\ell_1=0.30$ and~$\ell_2=0.20$ denote the side lengths of the rectangle. The parameter-dependent target is obtained by translating the rectangle according to the parameter, that is
\begin{align*}
    y_\mu^T(x) = \mathbf 1_{Q_\ell(\mu)}(x),
\end{align*}
where
\begin{align*}
    Q_\ell(\mu) = \left\{x=(x_1,x_2)\in\Omega\colon |x_1-\mu_1|\leq \frac{\ell_1}{2}\text{ and }|x_2-\mu_2|\leq \frac{\ell_2}{2}\right\}.
\end{align*}
The optimality system is given by
\begin{equation}\label{eq:optimality_global}
    \left\lbrace
    \begin{aligned}
        \partial_t y_\mu-\Delta y_\mu &= u_\mu && \text{in }(0,T)\times\Omega,\\
        -\partial_t \varphi_\mu-\Delta \varphi_\mu &= 0 && \text{in }(0,T)\times\Omega,\\
        \varphi_\mu(T) &= \alpha \big(y_\mu(T)-y_\mu^T\big) && \text{in }\Omega,\\
        u_\mu &= -\varphi_\mu && \text{in }(0,T)\times\Omega,
    \end{aligned}
    \right.
\end{equation}
together with homogeneous Dirichlet boundary conditions for both~$y_\mu$ and~$\varphi_\mu$ and zero initial condition for~$y_\mu$.
\par
For each training set size~$n_{\mathrm{train}}\in\{100,200,500,1000\}$, we generate snapshots by solving~\eqref{eq:optimality_global} for parameter values sampled uniformly from~$\params$.
These snapshots are used to train and evaluate the surrogate models described in the previous subsections.
We first study how the approximation accuracy scales with the number of training snapshots.
The corresponding average relative errors in the final time adjoint on a set of~$100$ uniform randomly drawn test parameters are reported in~\Cref{fig:rect_full-scaling}.
We remark here that for the autoencoders we tried different latent dimensions, namely~$8$, $16$, $32$ and~$64$, and only report the best results according to the validation loss of the parameter-to-latent regressor.
In this example, for~$n_\mathrm{train}\in\{100,200\}$ a latent dimension of~$N=16$ was chosen and for~$n_\mathrm{train}\in\{500,1000\}$ the validation loss indicated to use~$N=8$.
For the reduced space in the~RB-ROM we chose a dimension of~$64$, thus matching the largest latent dimension of the autoencoder.
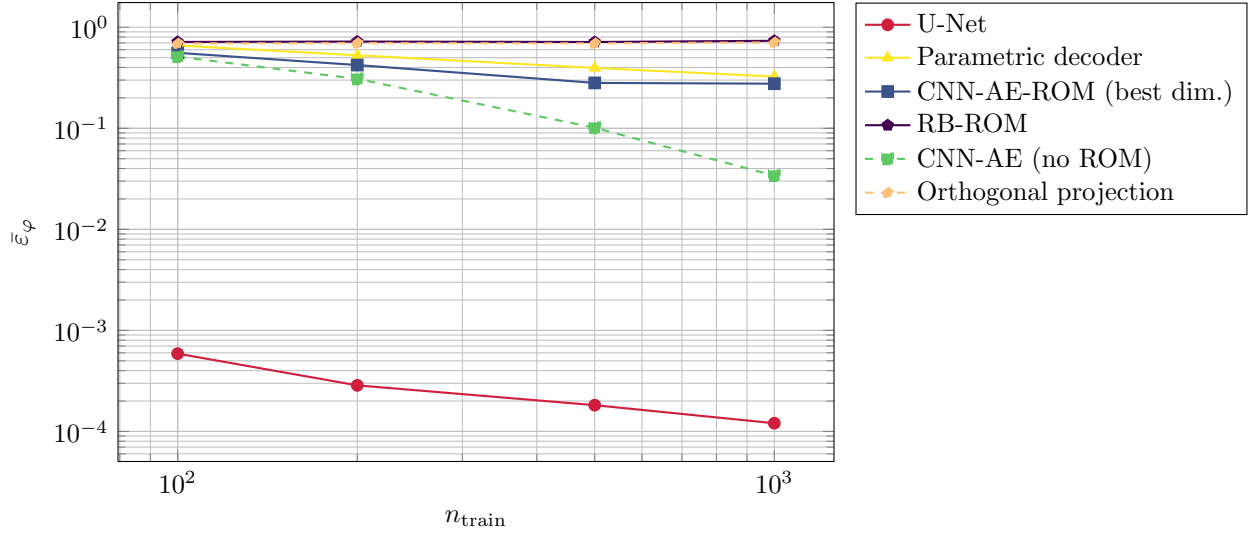
\begin{figure}[htbp]
    \centering
    \begin{tikzpicture}
        \begin{loglogaxis}[
            width=0.65\linewidth, height=0.45\linewidth,
            xlabel={$n_\mathrm{train}$}, ylabel={$\bar{\varepsilon}_\varphi$},
            legend pos=outer north east, legend cell align=left,
            grid=both, minor tick num=1,
            cycle list name=color list,
            ]
            \addplot[mark=*, thick, unet] table[x=n, y=unet_mean] {data/rect_full/scaling_curve.dat};
            \addlegendentry{U-Net}
            \addplot[mark=triangle*, thick, paramdec] table[x=n, y=param_dec_mean] {data/rect_full/scaling_curve.dat};
            \addlegendentry{Parametric decoder}
            \addplot[mark=square*, thick, cnnaerom] table[x=n, y=cnn_ae_rom_mean] {data/rect_full/scaling_curve.dat};
            \addlegendentry{CNN-AE-ROM (best dim.)}
            \addplot[mark=pentagon*, thick, rbrom] table[x=n, y=greedy_rb_rom_mean] {data/rect_full/scaling_curve.dat};
            \addlegendentry{RB-ROM}
            \addplot[mark=square*, thick, dashed, cnnae] table[x=n, y=cnn_ae_mean] {data/rect_full/scaling_curve.dat};
            \addlegendentry{CNN-AE (no ROM)}
            \addplot[mark=pentagon*, thick, dashed, orthproj] table[x=n, y=greedy_proj_mean] {data/rect_full/scaling_curve.dat};
            \addlegendentry{Orthogonal projection}
        \end{loglogaxis}
    \end{tikzpicture}
    \caption{Average relative error~$\bar{\varepsilon}_\varphi$ in the final time adjoint over~$100$ uniform randomly chosen test parameters depending on the training set size~$n_\mathrm{train}$ for the example with a moving rectangular target and global control domain.}
    \label{fig:rect_full-scaling}
\end{figure}
\par
\Cref{fig:rect_full-scaling} shows that the~U-Net provides the most accurate approximation of the final time adjoint for all training set sizes.
Already for~$n_{\mathrm{train}}=100$, the relative error is around~$6\cdot10^{-4}$, and the error decreases further to roughly~$10^{-4}$ as the number of training samples increases to~$n_{\mathrm{train}}=1000$.
This indicates that learning the solution operator from the parameter-dependent fields is particularly effective for this example and requires much less training data compared to the other approaches.
\par
The linear methods show a very different behavior.
Both the~RB-ROM and the orthogonal projection remain at errors of order~$10^{-1}$, essentially independently of~$n_{\mathrm{train}}$.
This is consistent with the transport-dominated structure of the target manifold.
Since the targets are translated discontinuous rectangles, the corresponding solution manifold is hard to approximate by low-dimensional linear spaces.
Hence, increasing the number of snapshots does not remove the main obstruction, which is the linearity of the approximation space.
\par
The autoencoder-based methods partially overcome this limitation.
In particular, the~CNN-AE without parameter-to-latent map improves significantly as more training data are added.
We emphasize, however, that this is not a surrogate model, since the autoencoder is applied directly to the true final time adjoint and therefore receives the exact latent representation.
Consequently, these results assess only the reconstruction capability of the learned nonlinear decoder.
The considerably larger errors of the~CNN-AE-ROM indicate that the main difficulty lies in predicting this latent representation from the parameter rather than reconstructing the solution once the latent variables are known.
Nevertheless, even when provided with the exact latent representation, the autoencoder remains less accurate than the~U-Net.
The parametric decoder also improves with the training set size, but its error stays relatively large.
This suggests that directly learning the map~$\mu\mapsto \varphi_\mu^*(T)$ does not exploit the full structure of the optimality system as effectively as the field-based~U-Net surrogate.
\par
While~\Cref{fig:rect_full-scaling} illustrates how the approximation accuracy evolves with the size of the training set, \Cref{tab:rect_full-summary} provides a quantitative comparison of the surrogate models.
Besides the relative error in the final time adjoint, we also report the corresponding errors in the recovered optimal state and optimal control, thereby assessing how the approximation quality propagates through the optimality system.
\begin{table}[htbp]
    \centering
    \begin{tabular}{llcccc}
        \toprule
        Method & Metric & $n_\mathrm{train}=100$ & $n_\mathrm{train}=200$ & $n_\mathrm{train}=500$ & $n_\mathrm{train}=1000$ \\
        \midrule\midrule
        U-Net & $\bar{\varepsilon}_\varphi$ & $5.89\cdot 10^{-4}$ & $2.86\cdot 10^{-4}$ & $1.82\cdot 10^{-4}$ & $1.21\cdot 10^{-4}$ \\
         & $\bar{\varepsilon}_y$ & $3.42\cdot 10^{-4}$ & $8.28\cdot 10^{-5}$ & $7.49\cdot 10^{-5}$ & $6.61\cdot 10^{-5}$ \\
         & $\bar{\varepsilon}_u$ & $8.54\cdot 10^{-4}$ & $3.14\cdot 10^{-4}$ & $2.13\cdot 10^{-4}$ & $1.60\cdot 10^{-4}$ \\
        \midrule
        Parametric decoder & $\bar{\varepsilon}_\varphi$ & $6.60\cdot 10^{-1}$ & $5.28\cdot 10^{-1}$ & $3.96\cdot 10^{-1}$ & $3.25\cdot 10^{-1}$ \\
         & $\bar{\varepsilon}_y$ & $5.67\cdot 10^{-2}$ & $3.57\cdot 10^{-2}$ & $2.47\cdot 10^{-2}$ & $2.07\cdot 10^{-2}$ \\
         & $\bar{\varepsilon}_u$ & $5.71\cdot 10^{-1}$ & $4.52\cdot 10^{-1}$ & $3.38\cdot 10^{-1}$ & $2.78\cdot 10^{-1}$ \\
        \midrule
        CNN-AE-ROM & $\bar{\varepsilon}_\varphi$ & $5.58\cdot 10^{-1}$ & $4.23\cdot 10^{-1}$ & $2.81\cdot 10^{-1}$ & $2.76\cdot 10^{-1}$ \\
         & $\bar{\varepsilon}_y$ & $3.67\cdot 10^{-2}$ & $2.68\cdot 10^{-2}$ & $2.80\cdot 10^{-2}$ & $2.81\cdot 10^{-2}$ \\
         & $\bar{\varepsilon}_u$ & $4.78\cdot 10^{-1}$ & $3.61\cdot 10^{-1}$ & $2.45\cdot 10^{-1}$ & $2.44\cdot 10^{-1}$ \\
        \midrule
        RB-ROM & $\bar{\varepsilon}_\varphi$ & $7.16\cdot 10^{-1}$ & $7.22\cdot 10^{-1}$ & $7.15\cdot 10^{-1}$ & $7.34\cdot 10^{-1}$ \\
         & $\bar{\varepsilon}_y$ & $1.69\cdot 10^{-2}$ & $1.77\cdot 10^{-2}$ & $1.71\cdot 10^{-2}$ & $1.81\cdot 10^{-2}$ \\
         & $\bar{\varepsilon}_u$ & $6.05\cdot 10^{-1}$ & $6.10\cdot 10^{-1}$ & $6.04\cdot 10^{-1}$ & $6.21\cdot 10^{-1}$ \\
        \midrule\midrule
        CNN-AE & $\bar{\varepsilon}_\varphi$ & $5.12\cdot 10^{-1}$ & $3.08\cdot 10^{-1}$ & $1.01\cdot 10^{-1}$ & $3.40\cdot 10^{-2}$ \\
         & $\bar{\varepsilon}_y$ & $3.49\cdot 10^{-2}$ & $2.13\cdot 10^{-2}$ & $1.37\cdot 10^{-2}$ & $1.03\cdot 10^{-2}$ \\
         & $\bar{\varepsilon}_u$ & $4.39\cdot 10^{-1}$ & $2.64\cdot 10^{-1}$ & $9.15\cdot 10^{-2}$ & $3.55\cdot 10^{-2}$ \\
        \midrule
        Orthogonal projection & $\bar{\varepsilon}_\varphi$ & $6.90\cdot 10^{-1}$ & $6.97\cdot 10^{-1}$ & $6.93\cdot 10^{-1}$ & $7.08\cdot 10^{-1}$ \\
         & $\bar{\varepsilon}_y$ & $6.05\cdot 10^{-2}$ & $5.44\cdot 10^{-2}$ & $5.32\cdot 10^{-2}$ & $6.27\cdot 10^{-2}$ \\
         & $\bar{\varepsilon}_u$ & $6.04\cdot 10^{-1}$ & $6.08\cdot 10^{-1}$ & $6.02\cdot 10^{-1}$ & $6.20\cdot 10^{-1}$ \\
        \bottomrule
    \end{tabular}
    \caption{Accuracy of the surrogates on the example with a moving rectangular target and global control domain.}
    \label{tab:rect_full-summary}
\end{table}
\par
The values in~\Cref{tab:rect_full-summary} confirm the trends observed in~\Cref{fig:rect_full-scaling}.
For~$n_{\mathrm{train}}=1000$, the~U-Net reaches an average relative error of~$1.21\cdot 10^{-4}$ for the final time adjoint, while the best autoencoder-based method without parameter-to-latent map gives~$3.4\cdot 10^{-2}$.
The linear reduced models remain between~$10^{-1}$ and~$1$ for the adjoint error.
\par
The errors for the state and the control follow the same qualitative behavior.
Since the optimal control is recovered from the adjoint, a more accurate approximation of~$\varphi_\mu^*(T)$ directly improves the control approximation.
The state errors are generally smaller than the adjoint and control errors, which is expected due to the smoothing effect of the heat equation dynamics.
\par
The previous results demonstrate that the proposed surrogate models exhibit substantially different approximation capabilities.
However, accuracy alone is not sufficient when assessing their suitability for multi-query optimal control problems.
Since the primary motivation for surrogate modeling is to reduce the computational costs of repeated queries, it is equally important to compare the offline training effort and the online evaluation time.
These quantities are reported in~\Cref{tab:rect_full-runtime}.
\begin{table}[htbp]
    \centering
    \begin{tabular}{l|cccc|c}
        \toprule
        \multirow{2}{*}{Method} & \multicolumn{4}{c|}{Offline training times} & \multirow{2}{*}{Online times} \\
         & $n_\mathrm{train}=100$ & $n_\mathrm{train}=200$ & $n_\mathrm{train}=500$ & $n_\mathrm{train}=1000$ & \\
        \midrule\midrule
        U-Net & 23.3min & 40.0min & 48.2min & 50.5min & 0.70ms \\
        Parametric decoder & 4.2min & 8.5min & 16.9min & 29.4min & 0.15ms \\
        CNN-AE-ROM & 39.4min & 52.9min & 1.9h & 3.5h & 0.21ms \\
        RB-ROM & 3.6h & 5.5h & 10.9h & 20.6h & 16337.3ms \\
        \midrule
        \multicolumn{5}{r}{FOM solve (per query)} & 5517.4ms \\
        \multicolumn{5}{r}{State and control recovery} & 255.7ms \\
        \bottomrule
    \end{tabular}
    \caption{Per-method wall-clock costs on the example with a moving rectangular target and global control domain. The online times reported here refer to the average time to compute the approximate final time adjoint. The full online query costs also include the time to obtain the state and control from the final time adjoint (shown in the last row).}
    \label{tab:rect_full-runtime}
\end{table}
\par
\Cref{tab:rect_full-runtime} shows that the neural network surrogates have negligible inference times compared with the full-order solve.
Among all methods, the~U-Net gives the best balance between accuracy and online efficiency.
\par
In contrast, the~RB-ROM is not competitive for this benchmark.
It has a large offline cost and an online time that is even larger than solving the~FOM.
Moreover, this additional cost does not lead to improved accuracy, because the approximation remains restricted to a linear reduced space.
\par
All methods require an additional solve of the adjoint dynamical system to obtain the approximate control from the predicted final time adjoint.
This part dominates in particular the computational costs of the machine learning approaches during the online phase.
We therefore report these costs separately from the pure prediction of the final time adjoint.
\par
We remark at this point that the a posteriori error estimator can be evaluated efficiently at costs similar to one state recovery, i.e.~around~$256$ms.
Moreover, we observed for this experiment an over-estimator factor of at most~$10.41$, leading to an error estimate that is at about one order of magnitude larger than the true error in the final time adjoint.
\paragraph{Main conclusions for the rectangular target example.}
Overall, this example illustrates the limitations of linear reduced-order models for transport-dominated manifolds with slow decay of the Kolmogorov~$N$-width.
Nonlinear latent representations improve the approximation, but the best performance is obtained by the~U-Net, which directly learns the field-dependent solution operator.
This confirms that exploiting the structure of the parametric components defining the optimality system is more effective than learning only a parameter-to-solution map in this case.

\subsection{Locally distributed control with moving targets and parametric diffusivity}\label{sec:ball-example}
As a second example, we consider a more challenging setting including a localized control operator and a parametric diffusivity that will be described below.
\par
Let~$\Omega = (0,1)^2$, $T= 1$ and let~$\omega = [0.2,0.8]^2 \subset \Omega$ be the control region. We consider the parameter domain
\begin{align*}
    \params = [0.3,0.7]\times [0.25,0.75]^2\subset \mathbb R^3.
\end{align*}
For each parameter~$(\mu_1,\mu_2,\mu_3)\in\params$, we consider the optimal control problem described by the heat equation with parametric diffusivity
\begin{equation*}
    \left\{
    \begin{aligned}
        \partial_t y_\mu - \nabla\cdot(\kappa_\mu \nabla y_\mu) &= \mathbf{1}_\omega u_\mu, && \text{in } (0,T)\times (\Omega)\\
        y_\mu &= 0, && \text{on }(0,T) \times \partial \Omega, \\
        y_\mu(0,x,y) &= 0, && \text{for }(x,y)\in\Omega,
    \end{aligned}
    \right.
\end{equation*}
where~$\kappa_\mu\colon\Omega\to\R$ is the parameter-dependent diffusion coefficient defined for~$x=(x_1,x_2)\in\Omega$ as
\begin{align*}
    \kappa_\mu(x)=\begin{cases}
        1,&\text{if }x_1\leq\mu_1,\\
        10,&\text{if }x_1>\mu_1.
    \end{cases}
\end{align*}
The objective is to steer the state towards a parameter-dependent target state~$y_\mu^T$ by minimizing the quadratic functional
\begin{align*}
    J_\mu(u) = \frac{\alpha}{2} \lVert y_\mu(T) - y_\mu^T \rVert_{L^2(\Omega)}^2 + \frac{1}{2} \lVert u\rVert_{L^2(0,T;\omega)}^2
\end{align*}
for~$\alpha=10000$.
We are going to introduce a transport-dominated family of targets, which is defined by
\begin{align*}
    y_\mu^T (x) = \mathbf 1_{B_{0.15}(\mu)}(x),\quad\text{where}\ B_{0.15}(\mu)=\left\{x\in\Omega: \lVert x-(\mu_2,\mu_3)\rVert_2\leq 0.15\right\}.
\end{align*}
For each parameter value, the optimality system is given by
\begin{equation}\label{eq:optimality_heat}
    \left\lbrace
    \begin{aligned}
        \partial_t y_\mu - \nabla\cdot(\kappa_\mu\nabla y_\mu) &= \mathbf{1}_\omega u_\mu && \text{in }(0,T)\times\Omega,\\
        -\partial_t \varphi_\mu - \nabla\cdot(\kappa_\mu\nabla \varphi_\mu) &= 0 && \text{in }(0,T)\times\Omega,\\
        \varphi_\mu(T) &= \alpha \big(y_\mu(T)-y_\mu^T\big) && \text{in }\Omega,\\
        u_\mu &= -\mathbf{1}_\omega \varphi_\mu && \text{in }(0,T)\times\omega,
    \end{aligned}
    \right.
\end{equation}
together with homogeneous Dirichlet boundary conditions for both~$y_\mu$ and~$\varphi_\mu$ and the initial condition~$y_\mu(0)=0$.
\par
As in the previous section, we generate for each training set size~$n_{\mathrm{train}}\in\{100,200,500,1000\}$ snapshots by solving~\eqref{eq:optimality_heat} for parameter values sampled from~$\params$ and collect them in the training set for the surrogate models.
We first illustrate the quality of the~U-Net predictions on representative parameters before comparing the different surrogate models quantitatively.
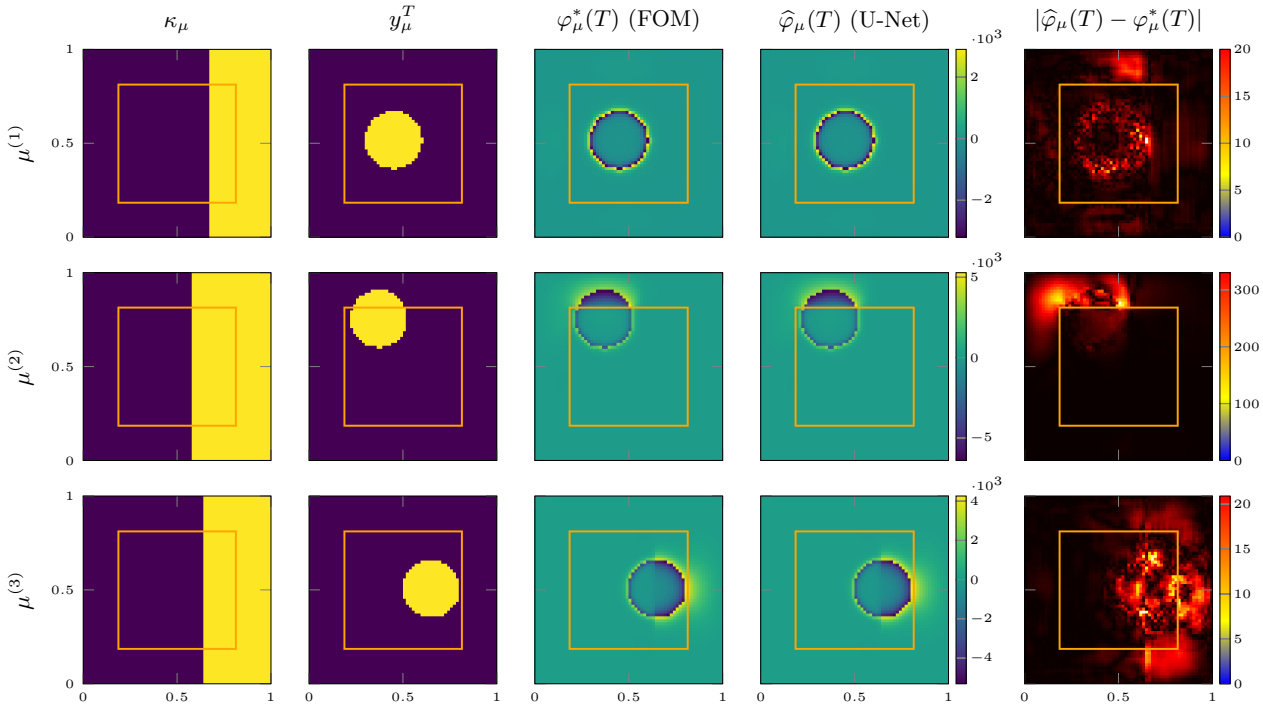
\begin{figure}[H]
  \centering
  \newcommand{\PhiMinA}{-3200}\newcommand{\PhiMaxA}{ 2900}\newcommand{\ErrMaxA}{ 20.0}
  \newcommand{\PhiMinB}{-6400}\newcommand{\PhiMaxB}{ 5280}\newcommand{\ErrMaxB}{330.2}
  \newcommand{\PhiMinC}{-5310}\newcommand{\PhiMaxC}{ 4260}\newcommand{\ErrMaxC}{ 20.9}
  \pgfplotsset{
    panel/.style={
      width=2.6cm, height=2.6cm,
      xmin=0, xmax=1, ymin=0, ymax=1,
      enlargelimits=false,
      axis on top, scale only axis,
      xtick={0,0.5,1}, ytick={0,0.5,1},
      tick label style={font=\tiny},
      title style={font=\small, yshift=-3pt},
    },
    cbar/.style={
      colorbar, colorbar style={width=4pt, font=\tiny, anchor=west, at={(1.04,0.5)}, yticklabel style={xshift=-2pt}}
    },
  }
  \resizebox{\textwidth}{!}{
  \begin{tikzpicture}
      \begin{groupplot}[
          group style={group size=5 by 3,
                       horizontal sep=15pt, vertical sep=14pt,
                       xticklabels at=edge bottom,
                       yticklabels at=edge left},
          panel,
      ]
      \nextgroupplot[title={$\kappa_\mu$},
                     ylabel={$\mu^{(1)}$},
                     ylabel style={align=center, font=\small}]
      \addplot graphics [xmin=0, xmax=1, ymin=0, ymax=1]
              {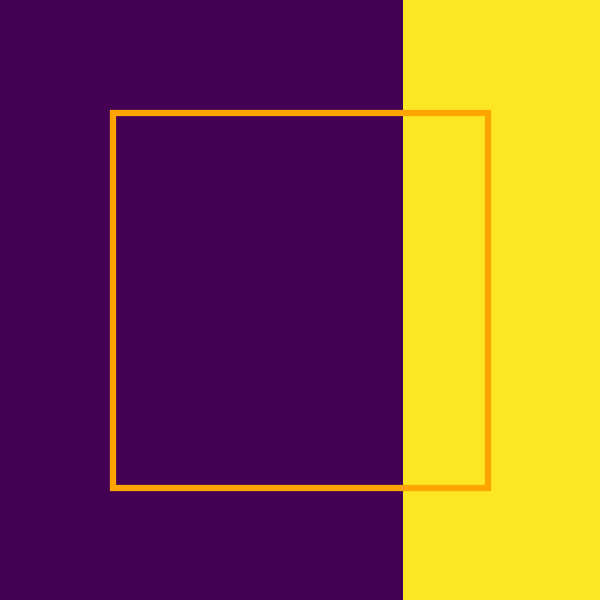};
      \nextgroupplot[title={$y_\mu^T$}]
      \addplot graphics [xmin=0, xmax=1, ymin=0, ymax=1]
              {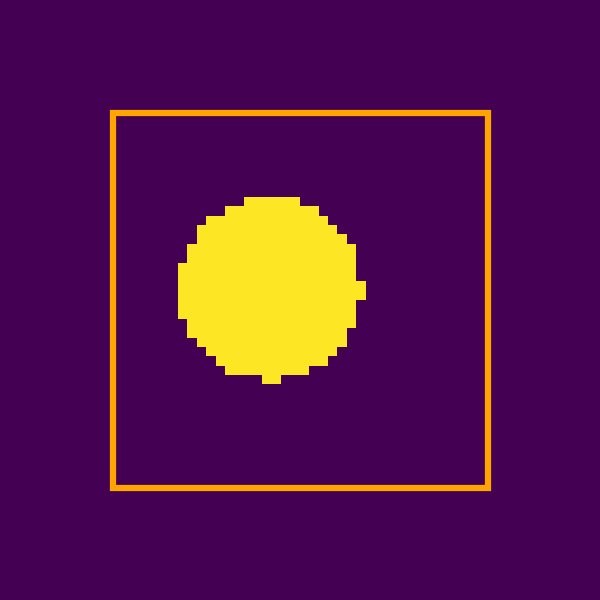};
      \nextgroupplot[title={$\varphi_\mu^*(T)$ (FOM)}]
      \addplot graphics [xmin=0, xmax=1, ymin=0, ymax=1]
              {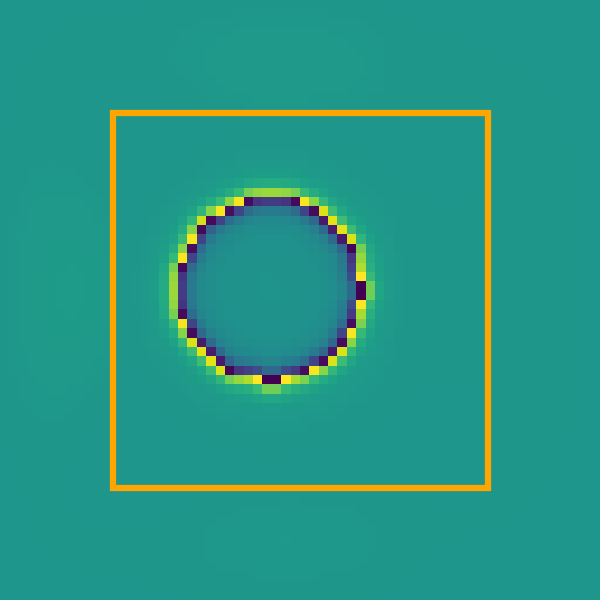};
      \nextgroupplot[title={$\widehat\varphi_\mu(T)$ (U-Net)},
                     colormap/viridis, cbar,
                     point meta min=\PhiMinA, point meta max=\PhiMaxA, colorbar style={scaled y ticks=base 10:-3, y tick scale label style={xshift=20pt}}]
      \addplot graphics [xmin=0, xmax=1, ymin=0, ymax=1]
              {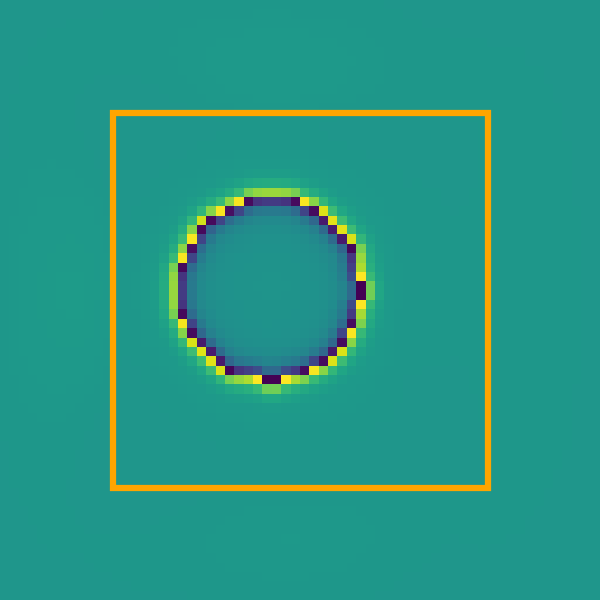};
      \nextgroupplot[title={$|\widehat\varphi_\mu(T) - \varphi_\mu^*(T)|$},
                     colormap/hot, cbar,
                     point meta min=0, point meta max=\ErrMaxA, xshift=15pt]
      \addplot graphics [xmin=0, xmax=1, ymin=0, ymax=1]
              {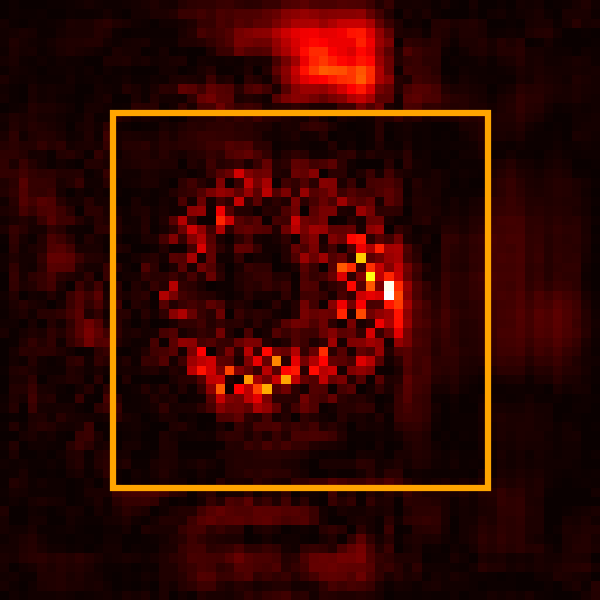};
    
      \nextgroupplot[ylabel={$\mu^{(2)}$},
                     ylabel style={align=center, font=\small}]
      \addplot graphics [xmin=0, xmax=1, ymin=0, ymax=1]
              {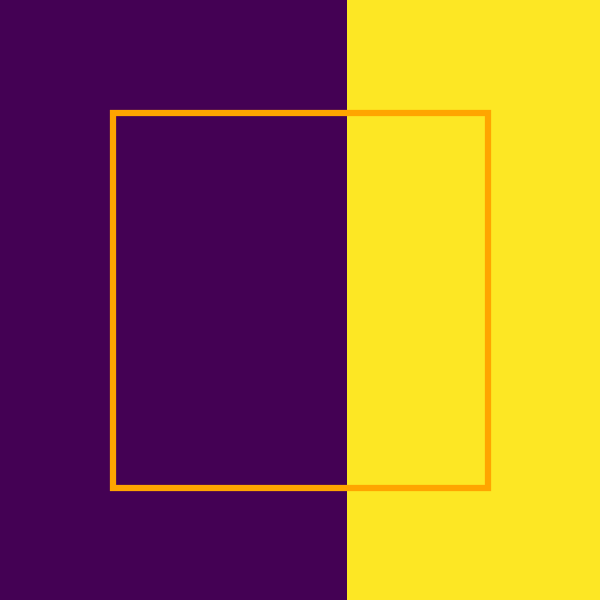};
      \nextgroupplot \addplot graphics [xmin=0, xmax=1, ymin=0, ymax=1]
              {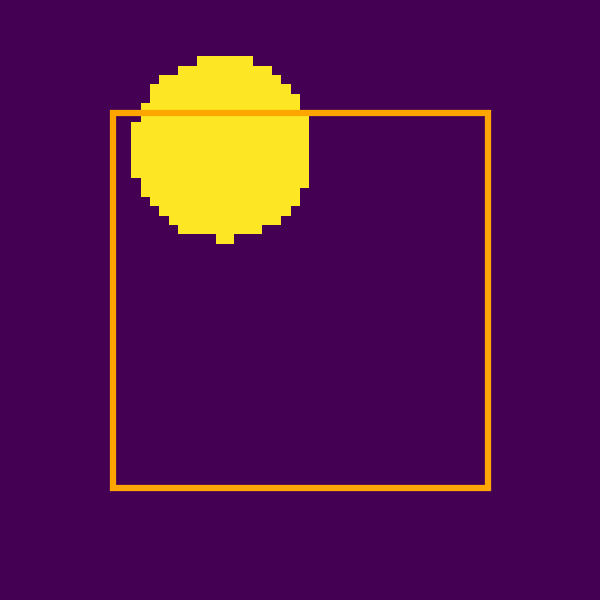};
      \nextgroupplot \addplot graphics [xmin=0, xmax=1, ymin=0, ymax=1]
              {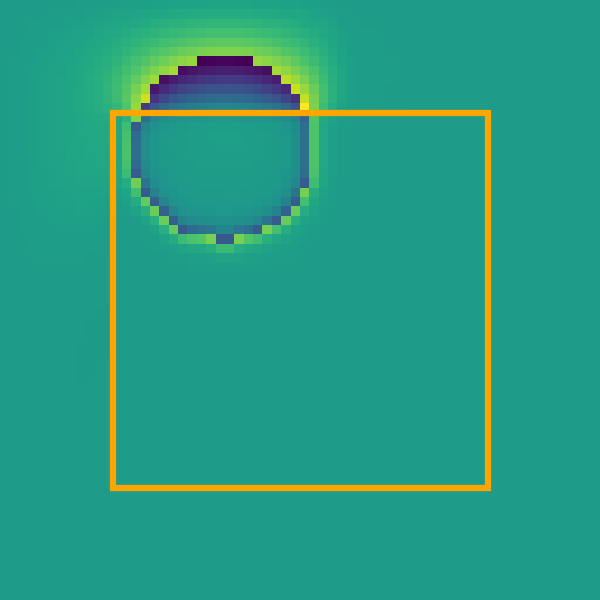};
      \nextgroupplot[colormap/viridis, cbar,
                     point meta min=\PhiMinB, point meta max=\PhiMaxB, colorbar style={scaled y ticks=base 10:-3, y tick scale label style={xshift=20pt}}]
      \addplot graphics [xmin=0, xmax=1, ymin=0, ymax=1]
              {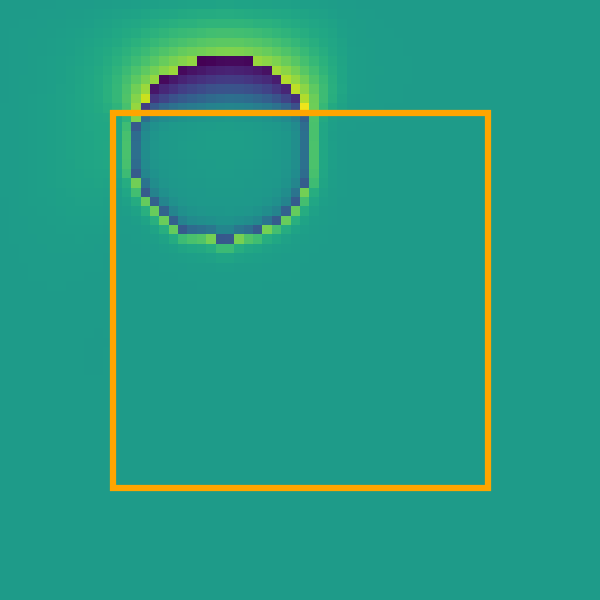};
      \nextgroupplot[colormap/hot, cbar,
                     point meta min=0, point meta max=\ErrMaxB]
      \addplot graphics [xmin=0, xmax=1, ymin=0, ymax=1]
              {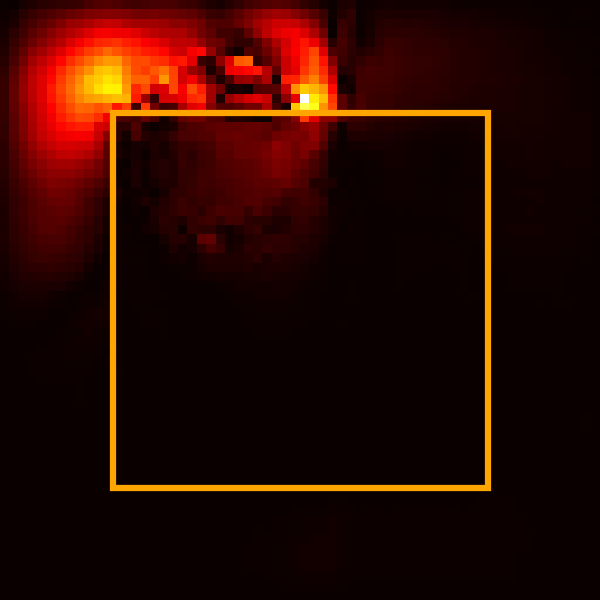};
    
      \nextgroupplot[ylabel={$\mu^{(3)}$},
                     ylabel style={align=center, font=\small}]
      \addplot graphics [xmin=0, xmax=1, ymin=0, ymax=1]
              {figures/ball_interface_joint_square/_mu208_kappa.png};
      \nextgroupplot \addplot graphics [xmin=0, xmax=1, ymin=0, ymax=1]
              {figures/ball_interface_joint_square/_mu208_xT.png};
      \nextgroupplot \addplot graphics [xmin=0, xmax=1, ymin=0, ymax=1]
              {figures/ball_interface_joint_square/_mu208_phi_fom.png};
      \nextgroupplot[colormap/viridis, cbar,
                     point meta min=\PhiMinC, point meta max=\PhiMaxC, colorbar style={scaled y ticks=base 10:-3, y tick scale label style={xshift=20pt}}]
      \addplot graphics [xmin=0, xmax=1, ymin=0, ymax=1]
              {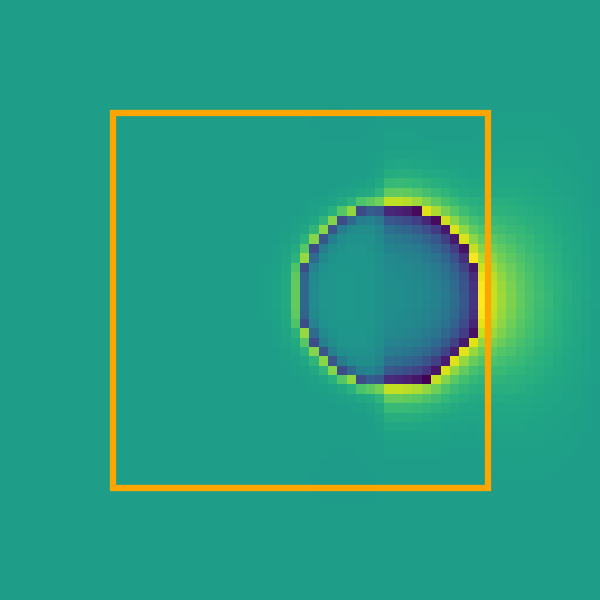};
      \nextgroupplot[colormap/hot, cbar,
                     point meta min=0, point meta max=\ErrMaxC]
      \addplot graphics [xmin=0, xmax=1, ymin=0, ymax=1]
              {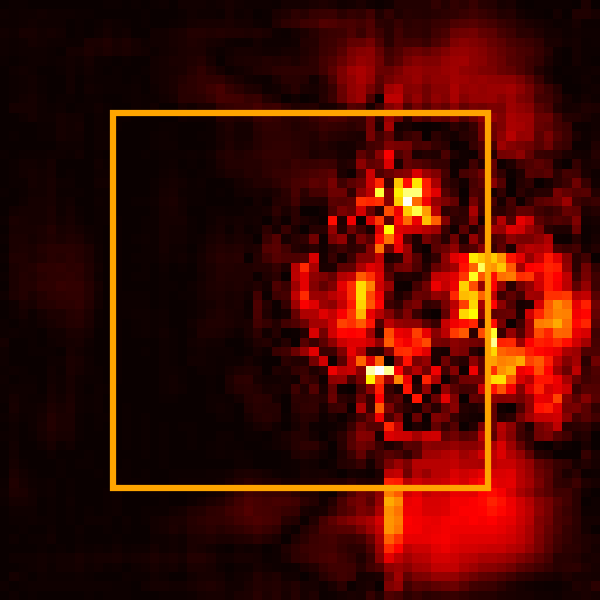};
      \end{groupplot}
  \end{tikzpicture}
  }
  \caption{Results of the U-Net surrogate for three test parameters~$\mu^{(1)}=[0.6708, 0.4508, 0.5183]$ (rel.~$L^2(\Omega)$-err.~in final time adjoint: $4.64\cdot10^{-3}$), $\mu^{(2)}=[0.5747, 0.3731, 0.7489]$ (rel.~$L^2(\Omega)$-err.~in final time adjoint: $4.88\cdot10^{-2}$) and~$\mu^{(3)}=[0.6416, 0.6476, 0.5066]$ (rel.~$L^2(\Omega)$-err.~in final time adjoint: $6.19\cdot10^{-3}$). Columns: diffusivity~$\kappa$, target~$y_\mu^T$, FOM~final time adjoint~$\varphi_\mu^*(T)$, U-Net prediction~$\widehat\varphi_\mu(T)$, and pointwise error~$|\widehat\varphi_\mu(T)-\varphi_\mu^*(T)|$. The orange square depicts the control domain~$\omega=[0.2,0.8]^2$.}
  \label{fig:ball-unet-panels}
\end{figure}
\par
\Cref{fig:ball-unet-panels} shows three parameter settings illustrating different situations for the diffusivity interface and target position.
In all cases, the predicted final time adjoint is visually almost indistinguishable from the corresponding full-order solution.
The largest discrepancies remain localized near the sharp transitions generated by the moving target and the diffusivity interface, while the overall spatial structure is captured accurately.
\par
The second row corresponds to the most challenging setting among the displayed examples (and actually among all test parameters), where the moving target crosses the boundary of the control region.
As is shown by the large relative error, the prediction becomes less accurate near the interaction between these two features.
Nevertheless, the~U-net still reproduces the correct global shape of the adjoint and preserves the main localization patterns.
Overall, these examples indicate that the proposed approach (and its architecture) is able to capture the influence of both the parameter-dependent diffusivity and the moving target on the optimality system.
\par
While~\Cref{fig:ball-unet-panels} illustrates the quality of the predictions for some selected parameters, we next analyze the overall approximation performance over the entire test set and its dependence on the number of available training snapshots.
For the autoencoders in this example, the latent dimension was selected as~$N=64$ for~$n_\mathrm{train}=200$ and as~$N=8$ for~$n_\mathrm{train}\in\{100,500,1000\}$.
\begin{figure}[htbp]
    \centering
    \begin{tikzpicture}
        \begin{loglogaxis}[
            width=0.65\linewidth, height=0.45\linewidth,
            xlabel={$n_\mathrm{train}$}, ylabel={$\bar{\varepsilon}_\varphi$},
            legend pos=outer north east, legend cell align=left,
            grid=both, minor tick num=1,
            cycle list name=color list,
            ]
            \addplot[mark=*, thick, unet] table[x=n, y=unet_mean] {data/ball_interface_joint_square/scaling_curve.dat};
            \addlegendentry{U-Net}
            \addplot[mark=triangle*, thick, paramdec] table[x=n, y=param_dec_mean] {data/ball_interface_joint_square/scaling_curve.dat};
            \addlegendentry{Parametric decoder}
            \addplot[mark=square*, thick, cnnaerom] table[x=n, y=cnn_ae_rom_mean] {data/ball_interface_joint_square/scaling_curve.dat};
            \addlegendentry{CNN-AE-ROM (best dim.)}
            \addplot[mark=pentagon*, thick, rbrom] table[x=n, y=greedy_rb_rom_mean] {data/ball_interface_joint_square/scaling_curve.dat};
            \addlegendentry{RB-ROM}
            \addplot[mark=square*, thick, dashed, cnnae] table[x=n, y=cnn_ae_mean] {data/ball_interface_joint_square/scaling_curve.dat};
            \addlegendentry{CNN-AE (no ROM)}
            \addplot[mark=pentagon*, thick, dashed, orthproj] table[x=n, y=greedy_proj_mean] {data/ball_interface_joint_square/scaling_curve.dat};
            \addlegendentry{Orthogonal projection}
        \end{loglogaxis}
    \end{tikzpicture}
    \caption{Average relative error~$\bar{\varepsilon}_\varphi$ in the final time adjoint over~$100$ uniform randomly chosen test parameters depending on the training set size~$n_\mathrm{train}$ for the example with a moving ball target, diffusivity with a moving interface and local control domain.}
    \label{fig:ball_interface_joint_square-scaling}
\end{figure}
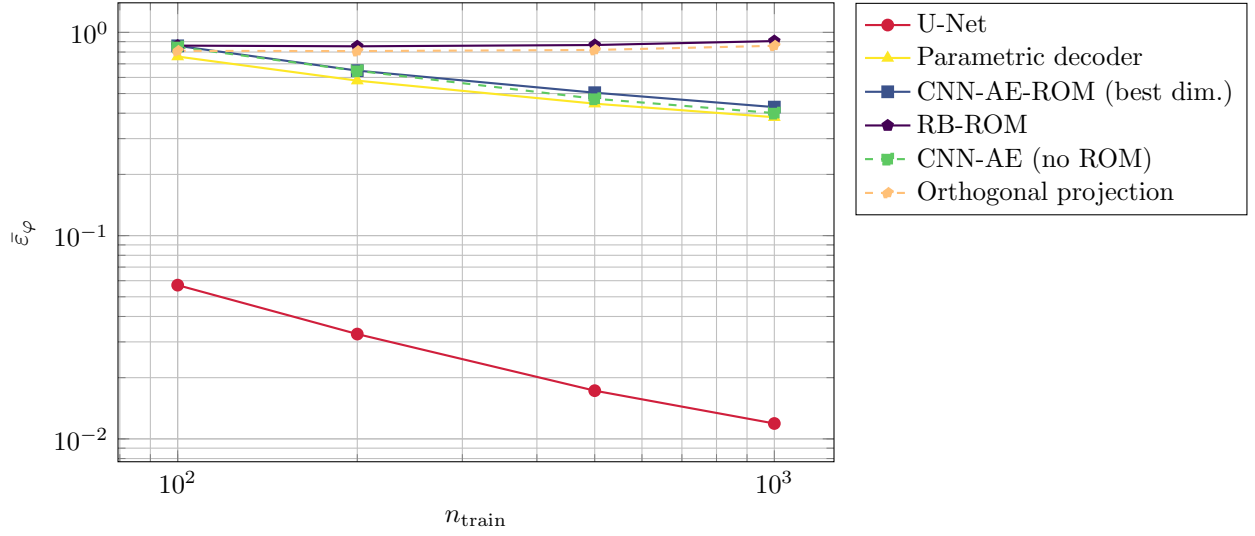
\par
The scaling behavior shown in~\Cref{fig:ball_interface_joint_square-scaling} is consistent with the observations from the previous example in~\Cref{ssec:HGC}, although the present problem is considerably more challenging due to the simultaneous presence of a moving target, a parameter-dependent diffusivity, and a localized control region.
The~U-Net consistently achieves the smallest approximation errors and exhibits a clear improvement as the number of training snapshots increases.
\par
The remaining surrogate models perform substantially worse.
The~RB-ROM and the orthogonal projection remain essentially insensitive to the size of the training set, reflecting, again, the limitations of linear approximation spaces for transport-dominated solution manifolds.
The autoencoder-based methods improve as additional training data becomes available, but the gain is considerably smaller than for the~U-Net.
This suggests that nonlinear latent representations alone are not sufficient to capture the increased complexity introduced by the parameter-dependent diffusion coefficient and the localized control and that a substantial amount of additional training data would be required to achieve acceptable approximation errors.
\par
The behavior observed in~\Cref{fig:ball_interface_joint_square-scaling} is quantified in~\Cref{tab:ball_interface_joint_square-summary}, which reports the average relative errors for the final time adjoint together with the corresponding errors in the recovered state and optimal control.
\begin{table}[htbp]
    \centering
    \begin{tabular}{llcccc}
        \toprule
        Method & Metric & $n_\mathrm{train}=100$ & $n_\mathrm{train}=200$ & $n_\mathrm{train}=500$ & $n_\mathrm{train}=1000$ \\
        \midrule\midrule
        U-Net & $\bar{\varepsilon}_\varphi$ & $5.71\cdot 10^{-2}$ & $3.28\cdot 10^{-2}$ & $1.73\cdot 10^{-2}$ & $1.19\cdot 10^{-2}$ \\
         & $\bar{\varepsilon}_y$ & $3.70\cdot 10^{-2}$ & $2.36\cdot 10^{-2}$ & $1.20\cdot 10^{-2}$ & $9.31\cdot 10^{-3}$ \\
         & $\bar{\varepsilon}_u$ & $5.53\cdot 10^{-2}$ & $3.15\cdot 10^{-2}$ & $1.64\cdot 10^{-2}$ & $1.14\cdot 10^{-2}$ \\
        \midrule
        Parametric decoder & $\bar{\varepsilon}_\varphi$ & $7.60\cdot 10^{-1}$ & $5.79\cdot 10^{-1}$ & $4.46\cdot 10^{-1}$ & $3.82\cdot 10^{-1}$ \\
         & $\bar{\varepsilon}_y$ & $1.55\cdot 10^{-1}$ & $1.14\cdot 10^{-1}$ & $7.86\cdot 10^{-2}$ & $6.39\cdot 10^{-2}$ \\
         & $\bar{\varepsilon}_u$ & $8.12\cdot 10^{-1}$ & $6.30\cdot 10^{-1}$ & $4.88\cdot 10^{-1}$ & $4.22\cdot 10^{-1}$ \\
        \midrule
        CNN-AE-ROM & $\bar{\varepsilon}_\varphi$ & $8.58\cdot 10^{-1}$ & $6.48\cdot 10^{-1}$ & $5.05\cdot 10^{-1}$ & $4.29\cdot 10^{-1}$ \\
         & $\bar{\varepsilon}_y$ & $1.48\cdot 10^{-1}$ & $1.11\cdot 10^{-1}$ & $7.45\cdot 10^{-2}$ & $6.47\cdot 10^{-2}$ \\
         & $\bar{\varepsilon}_u$ & $8.96\cdot 10^{-1}$ & $6.91\cdot 10^{-1}$ & $5.50\cdot 10^{-1}$ & $4.68\cdot 10^{-1}$ \\
        \midrule
        RB-ROM & $\bar{\varepsilon}_\varphi$ & $8.60\cdot 10^{-1}$ & $8.54\cdot 10^{-1}$ & $8.66\cdot 10^{-1}$ & $9.06\cdot 10^{-1}$ \\
         & $\bar{\varepsilon}_y$ & $3.66\cdot 10^{-2}$ & $3.52\cdot 10^{-2}$ & $3.64\cdot 10^{-2}$ & $3.80\cdot 10^{-2}$ \\
         & $\bar{\varepsilon}_u$ & $8.98\cdot 10^{-1}$ & $8.79\cdot 10^{-1}$ & $8.90\cdot 10^{-1}$ & $9.26\cdot 10^{-1}$ \\
        \midrule\midrule
        CNN-AE & $\bar{\varepsilon}_\varphi$ & $8.47\cdot 10^{-1}$ & $6.47\cdot 10^{-1}$ & $4.71\cdot 10^{-1}$ & $4.01\cdot 10^{-1}$ \\
         & $\bar{\varepsilon}_y$ & $1.42\cdot 10^{-1}$ & $1.04\cdot 10^{-1}$ & $9.30\cdot 10^{-2}$ & $6.75\cdot 10^{-2}$ \\
         & $\bar{\varepsilon}_u$ & $8.90\cdot 10^{-1}$ & $6.92\cdot 10^{-1}$ & $5.26\cdot 10^{-1}$ & $4.37\cdot 10^{-1}$ \\
        \midrule
        Orthogonal projection & $\bar{\varepsilon}_\varphi$ & $8.11\cdot 10^{-1}$ & $8.09\cdot 10^{-1}$ & $8.19\cdot 10^{-1}$ & $8.59\cdot 10^{-1}$ \\
         & $\bar{\varepsilon}_y$ & $1.05\cdot 10^{-1}$ & $1.03\cdot 10^{-1}$ & $1.06\cdot 10^{-1}$ & $1.06\cdot 10^{-1}$ \\
         & $\bar{\varepsilon}_u$ & $8.64\cdot 10^{-1}$ & $8.55\cdot 10^{-1}$ & $8.67\cdot 10^{-1}$ & $9.01\cdot 10^{-1}$ \\
        \bottomrule
    \end{tabular}
    \caption{Accuracy of the surrogates on the example with a moving ball target, diffusivity with a moving interface and local control domain.}
    \label{tab:ball_interface_joint_square-summary}
\end{table}
\par
\Cref{tab:ball_interface_joint_square-summary} confirms the superior performance of the~U-Net across all training set sizes.
For~$n_{\mathrm{train}}=1000$, the average relative error in the final time adjoint is reduced to~$1.19\cdot10^{-2}$, whereas the remaining surrogate models still exhibit errors between approximately~$4\cdot10^{-1}$ and~$9\cdot10^{-1}$.
Although the absolute errors are larger than in the previous benchmark, this is expected due to the additional complexity introduced by the discontinuous diffusivity and the localized control operator.
As in the previous example in~\Cref{ssec:HGC}, the state and control errors follow the same qualitative behavior as the adjoint errors.
\par
Besides approximation accuracy, computational efficiency is an essential criterion in multi-query optimal control.
We therefore conclude the comparison by reporting the offline training cost and the online evaluation time for each surrogate model.
\begin{table}[htbp]
    \centering
    \begin{tabular}{l|cccc|c}
        \toprule
        \multirow{2}{*}{Method} & \multicolumn{4}{c|}{Offline training times} & \multirow{2}{*}{Online times} \\
         & $n_\mathrm{train}=100$ & $n_\mathrm{train}=200$ & $n_\mathrm{train}=500$ & $n_\mathrm{train}=1000$ & \\
        \midrule\midrule
        U-Net & 26.2min & 33.2min & 36.1min & 1.2h & 0.69ms \\
        Parametric decoder & 4.8min & 7.9min & 17.3min & 30.8min & 0.15ms \\
        CNN-AE-ROM & 35.1min & 55.2min & 2.2h & 4.9h & 0.22ms \\
        RB-ROM & 3.6h & 5.5h & 10.9h & 20.6h & 16418.8ms \\
        \midrule
        \multicolumn{5}{r}{FOM solve (per query)} & 17141.0ms \\
        \multicolumn{5}{r}{State and control recovery} & 256.3ms \\
        \bottomrule
    \end{tabular}
    \caption{Per-method wall-clock costs on the example with a moving ball target, diffusivity with a moving interface and local control domain. The online times reported here refer to the average time to compute the approximate final time adjoint. The full online query costs also include the time to obtain the state and control from the final time adjoint (shown in the last row).}
    \label{tab:ball_interface_joint_square-runtime}
\end{table}
\par
The computational costs reported in~\Cref{tab:ball_interface_joint_square-runtime} closely resemble those observed in the previous example.
All neural network surrogates provide inference times below one millisecond, and even after including the computation of state and control, the machine learning-based approaches lead to speedups around~$67$ compared to the~FOM.
\par
The~RB-ROM requires substantially larger offline and online computational effort while delivering considerably lower approximation accuracy.
Consequently, for this more challenging benchmark, the~U-Net again provides the most favorable balance between computational costs and accuracy.
\paragraph{Main conclusions for the moving interface, localized control example.}
This example shows that the proposed~U-Net surrogate remains accurate even when several parameter-dependent mechanisms are present simultaneously, namely a moving target, a discontinuous diffusivity with moving interface, and a localized control region.
Compared to the previous benchmark, all surrogate models experience a loss of accuracy, showing the increased complexity of the solution manifold.
Nevertheless, the~U-Net clearly beats the competing approaches, indicating that learning the solution operator from the parameter-dependent fields provides a better and more expressive approximation than learning low-dimensional representations or parameter-to-solution maps.

\section{Conclusion and outlook}\label{sec:conclusion-outlook}
\paragraph{Problem statement and theoretical findings.}
This work was concerned with parameter-dependent optimal control problems.
We considered a specific class of problems for which the application of the~HUM method reduces the optimality system to a parameter-dependent linear system for the optimal final-time adjoint.
Solving this system for many different parameter values is prohibitively costly.
Therefore, in many-query or real-time scenarios, suitable surrogate models are required.
\par
A first objective of this paper was to understand the approximation properties of the manifold of parameter-dependent final time adjoints.
We showed that even for systems that present regularizing effects (as the heat equation), there are parameterizations for which the Kolmogorov~$N$-widths of the manifold decay slowly, for instance polynomially.
More precisely, for distributed control of the heat equation with a constant diffusivity field, we proved that the Kolmogorov~$N$-width decay of the manifold of parameter-dependent target states is transferred to the manifold of final time adjoints.
Consequently, whenever the target manifold exhibits a slow Kolmogorov~$N$-width decay, the same limitation is inherited by the corresponding family of final time adjoints.
\paragraph{Nonlinear surrogate models.}
The slow decay of the Kolmogorov width explains why linear reduced-order models require increasingly large reduced spaces to achieve accurate approximations.
To overcome this limitation, we proposed a nonlinear surrogate based on a~U-Net architecture that learns the mapping from parameter-dependent fields, such as target states or diffusivity fields, to the corresponding final time adjoint.
Since the surrogate approximates the final time adjoint itself, the residual-based a posteriori error estimator, originally proposed for reduced basis methods, remains applicable, providing an efficient certification of the approximation quality.
\paragraph{Numerical validation.}
We performed two numerical experiments in order to validate our approach and to compare it to other methods, such as reduced basis models, parametric decoders or autoencoder and parameter-to-latent surrogates.
The first example corresponds to the setting analyzed in the theoretical part of the paper, while the second one considers a more challenging problem involving a localized control region, a discontinuous parameter-dependent diffusion coefficient, and moving target states.
In both cases, the numerical results confirm the theoretical findings: linear reduced-order models suffer from the slow Kolmogorov width decay and require large reduced spaces to achieve a satisfactory accuracy.
Nonlinear approaches such as parametric decoders or the convolutional autoencoders with parameter-to-latent map perform slightly better in terms of accuracy than the reduced basis model but require a large amount of training data.
\par
Among all tested methods, the~U-Net consistently produced the most accurate approximations while requiring significantly fewer training samples.
In particular, using only~100 training snapshots, the U-Net achieved average final time adjoint errors more than one order of magnitude smaller than the competing approaches trained on~1000 snapshots.
Moreover, after the offline training phase, the online computational cost remained substantially smaller than that of the full-order model, yielding speedups of approximately~67 in the most challenging example.
\par
Summarizing, for the problems we considered in this work, the~U-Net approach seems preferable among the tested surrogates in terms of accuracy and efficient usage of training data.
\paragraph{Future directions.}
An important theoretical direction is the analysis of approximation properties of~U-Nets for parameter-dependent optimal control problems, both for the class of problems considered here and for more general settings.
Building upon the recent approximation results in~\cite{schuette2024multilevel}, it would be of particular interest to establish expressivity results for learning the operator mapping parameter-dependent fields to optimal final time adjoints. 
\par
From the numerical point of view, several extensions remain open.
The current architecture assumes rectangular computational domains and regular grids; more general geometries and discretizations may require graph-based neural network architectures such as~Graph~U-Nets~\cite{gao2021graph}.
Another natural extension concerns time-dependent problems.
Incorporating time as an additional network dimension through a~3D~U-Net, or combining the architecture with temporal compression or attention mechanisms, may allow us to deal with time-varying coefficients.
Furthermore, as discussed in~\Cref{sec:discussion-nonlinear-approach}, extending the methodology to parameter-dependent boundary conditions remains an open challenge.
\par
Finally, investigating the performance of the proposed~U-Net approach on the elliptic benchmark problems considered in~\cite{franco2023deep} constitutes a natural next step.
Since the approach based on~U-Nets is directly applicable in that setting, such a comparison would provide further insights into the advantages and limitations of field-based surrogates relative to autoencoder and parameter-to-latent approaches.

\section*{Statements and declarations}
\paragraph{Funding}
This article is based upon work from COST Action InterCoML, CA24136, supported by COST (European Cooperation in Science and Technology).
The authors (ML) and (JRM) were partially supported by the Croatian Science Foundation under the project Conduction \textit{``Optimal control and model reduction for evolution and data-driven problems''} IP-2022-10-5191 and by the European Union -- NextGenerationEU through the institutional project TRIUMPh at the University of Dubrovnik. Additionally, (JRM) was also co-supported by the FONDECYT project 3260141 \textit{``Control and long-time behavior for dispersive equations''}.

\paragraph{Competing interests}
The authors have no relevant financial or non-financial interests to disclose.

\paragraph{Code availability}
The source code used to perform the experiments shown in this paper is available in~\cite{sourcecode}.

\printbibliography

\end{document}